\documentclass[11pt,openany,leqno]{article}
\usepackage[backref=page,colorlinks=true]{hyperref}
\usepackage{amsmath,amsthm,amsfonts,amssymb,amscd,hyperref}
\usepackage[capitalize]{cleveref}
\usepackage{tikz-cd} 
\usepackage{graphics,xcolor,url}
\numberwithin{equation}{section}
\input{xy} \xyoption{all}
\usepackage[latin1]{inputenc}
\usepackage{enumerate}
\usepackage{hyperref}
\usepackage{epsfig} 
\usepackage{faktor}
\usepackage{xfrac}
\usepackage{float}
\input{xy} \xyoption{all} 
\usepackage{mathrsfs}  

\def \cl{\mathrm{cl}}
\def \exp{\mathrm{exp}}
\def \id{\mathrm{id}}
\def \Id{\mathrm{id}}

\def \ad{\mathrm{ ad}}
\def \Ad{ \mathrm {Ad\,}}

\def \cN{{\mathcal N}}
\def \cI{{\mathcal I}}

\def \cP{{\mathscr P}}

\def \cV{{\mathcal V}}

\def \cQ{{\mathcal Q}}

\def\qand{\quad \text{and}\quad}
\def\diff{\mathfrak{diff}}
\def\Diff{\mathrm{Diff}}

\def\Ad{\mathrm{Ad}}

\def\B{\mathbb B}
\def\bB{\mathbb B}

\def\G{\mathbb G}
\def\P{\mathbb{P}}

\def\R{\mathbb R}
\def\N{\mathbb N}

\def\T{\mathbb T}

\def\Z{\mathbb Z}

\def\id{\mathrm{id}}

\def\supp{\mathrm{supp}\:}

\def\sg{{\mathfrak g}}
\def\si{{\mathfrak i}}
\def\sh{{\mathfrak h}}

\def\seig{{\mathfrak {Eig}}}

\def\spp{{\mathfrak{p}}}

\def \Fl{\mathrm{Fl}}
\usepackage[all]{xy}

\newtheorem{proposition}{Proposition}[section]

\newtheorem*{theorem*}{Theorem}
\newtheorem{coro}[proposition]{Corollary}

\newtheorem{problem}[proposition] {Problem}
\newtheorem*{problem*}{Problem}

\newtheorem{lemma}[proposition] {Lemma}

\newtheorem{sublemma}[proposition] {Sub-Lemma}
\newtheorem{theo}{Theorem}

\newtheorem{theoprime}{Theorem}

\newtheorem{fact}[proposition]{Fact}

\newtheorem{corollary}[theo]{Corollary}

\newtheorem{question}[proposition]{Question}
\newtheorem{conjecture}[proposition]{Conjecture}

\theoremstyle{remark}
\newtheorem{example}[proposition]{Example}
\newtheorem{remark}[proposition]{Remark}

\theoremstyle{definition}
\newtheorem{definition}[proposition] {Definition}
\DeclareTextFontCommand{\emph}{\em\bf}

\begin{document}

 \title{Every diffeomorphism is a total renormalization of a close to identity map}
\author{Pierre Berger\thanks{IMJ-PRG, CNRS, Sorbonne University, Paris University, partially supported by the ERC project 818737 Emergence of wild differentiable dynamical systems.
},  Nicolaz Gourmelon\thanks{IMB, UMR 5251, Universit\'e de Bordeaux.}, Mathieu Helfter\thanks{IMJ-PRG, CNRS, Sorbonne University, Paris University, partially supported by the ERC project 818737 Emergence of wild differentiable dynamical systems.
}
}
\date{\today}
\maketitle
\abstract{
For any $1\le r\le \infty$, we show that every  diffeomorphism of a manifold of the form $\R/\Z \times M$ is a total renormalization of a $C^r$-close to identity map. In other words, 
for every diffeomorphism $f$ of $\R/\Z \times M$, there exists a map $g$ arbitrarily close to identity such that the first return map of $g$ to a domain is conjugate to  $f$ and moreover 
the orbit of  this  domain  is equal to 
$\R/\Z\times M$. This enables us to localize near the identity the existence of many properties in dynamical systems, such as being Bernoulli for a smooth volume form. }
\tableofcontents
\section*{Introduction}
\subsection{Statements of the main theorems}
Let $\bB^n$ be the unit closed ball of $\R^n$. 

\begin{definition}[Primitive renormalization] \label{def: renormalization}
A {\em primitive renormalization} $G$ of a diffeomorphism $g\in \Diff(\bB^n) $  is a rescaling of an iteration of $g$. In other words, there exists $N \geq2 $ and an  embedding  $\psi:  \bB^n\hookrightarrow  \bB^n$  such that  $ g^i ( \psi ( \B^n) ) \cap \psi(\B^n) = \emptyset $ for every $ 0 < i < N $ and: 
\begin{equation*}  G = \psi^{-1} \circ g^N \circ \psi .\end{equation*} 
%The map $\psi^{-1}$ is called the {\em rescaling map} of the renormalization.
\end{definition}

A long standing open problem of dynamical  systems theory is:
\begin{problem}[1971]\label{mainproblem}
Which dynamics can be reached by renormalization of close to identity maps?
\end{problem}
This problem was first studied by Ruelle and Takens in  \cite{RT71}. Motivated by the study of turbulence, they proved that for any integer $n \ge 2 $, any dynamics on the $n$-dimensional torus is the renormalization of a $C^n$-close to identity map. This enabled them to construct perturbations of the identity map  of the torus with a strange attractor. Based on this, they conjectured that this appears as well in fluid dynamics and could be used as a mathematical definition of the notion of turbulence \cite{Lo63}. 

The main mathematical issue with this result is that the regularity is   limited~by the dimension of the torus.  However when considering flows, this problem was solved by Newhouse, Ruelle and Takens in \cite{NRT78}: given any vector filed $X$  equal to a rotation  on the torus $\T^n$, $n\ge 3$ and any map $F_0\in \Diff^\infty(\T^{n-1})$ homotopic to the identity, 
they perturbed $X$ to $ \widetilde{X} $ so that
its first return map to a global transverse section is $F_0$. 
Yet $\Diff^\infty(\T^{n})$ is ``much larger'' than $\Diff^\infty(\T^{n-1})$ and so the mathematical \cref{mainproblem} remains unsolved. 

A breakthrough was then performed in the seminal work of Turaev who proved that a $C^r$-\emph{dense subset} of $C^r$-orientation preserving embeddings of $ \bB^n$  could be obtained after renormalization of an arbitrarily close to identity map, for every $0\le r\le \infty$. 

A first main result is a solution to   \cref{mainproblem}, where  we improve Turaev's theorem to obtain, via a self-contained and new proof, \emph{any} $C^r$-orientation preserving map of $ \bB^n$ (instead of maps among a dense subset): 
 \begin{theo} 
\label{main1}For any $1\le r\le \infty$ and any orientation preserving $G\in \Diff^r(  \bB^n)$, in any neighborhood  $\cN \subset \Diff^r (\bB^n)$ of the identity,  there exists $g\in \cN$  such that a primitive renormalization of $g$ is   \textbf{equal} to $G$.  Moreover the rescaling map of this renormalization can be chosen affine.  
\end{theo}
A natural open problem is whether $g$ can be obtained conservative or symplectic when $G$ is conservative or symplectic. In this direction let us mention the work of Gonchenko-Shilnikov-Turaev \cite{GST07} who proved that, for every $0\le r\le \infty$, a $C^r$-dense subset of volume preserving embeddings of $ \bB^2$
could be obtained   after renormalization of an arbitrarily close to identity volume preserving map. 
Recently Fayad and Saprykina in \cite{fayad2021realizing} showed that any conservative map of the $n$-dimensional ball can be realized   by renormalized iteration  of a conservative $C^n$-perturbation of the identity.

If all these theorems indicate the richness of the possible  dynamical behaviors near the identity, one can object  the following. In the setting of \cref{def: renormalization}, the orbit of  $\bigcup_{k=0}^{N-1}  \psi(\bB^n)$ of the renormalization domain might be extremely small and so experimentally not observable. This objection is lifted completely when the renormalization domain intersects every orbit. This leads us to   generalize the notion of renormalization by the following:

\begin{definition}[Renormalization] \label{def: renormalization2} Let $r\in \{1, \dots, \infty\}\cup \{\omega \}$ and let $V$ be a manifold (with boundary). A  map $g\in \Diff^r(V)$  is {\em renormalizable} if there exists a strict submanifold with corners   $\Delta\subsetneq V$ such that: 
\begin{itemize}
\item there exists a bijective, local $C^r$-diffeomorphism $H: \Delta \to V$, called the {\em rescaling map} of the {\em renormalization domain} $\Delta$,
\item the first return time $\tau :\Delta \to \N^* $  into $\Delta$ by $g$ is bounded and the {\em renormalization} $G=H\circ g^\tau\circ H^{-1}$ belongs to $\Diff^r(V)$. \end{itemize} 
The map $g$ is {\em totally renormalizable} if the forward orbit of $\Delta$ covers $V$, i.e. $\bigcup_{n\ge 0} g^n(\Delta)=V $.  The map $G$ is then a {\em total renormalization} of $g$. 
\end{definition} 
 
\begin{remark}
Note that if $g\in \Diff^r(\B^n)$ displays a  primitive renormalization with embedding $\psi: \B^n\hookrightarrow \B^n$ and time $N$, then 
$\Delta:=\psi(\B^n)$ is renormalization domain of $g$ with constant return time $\tau\equiv N$ and rescaling map $H=\psi^{-1}$. Hence \cref{def: renormalization2} generalizes  
\cref{def: renormalization}. Note that the latter renormalization is never total. 

Moreover \cref{def: renormalization2} allows to consider a larger class of manifolds $V$ as we do not ask $ H = \psi^{-1} $ to be continuous on the boundary of the renormalization domain. The next example is about a total renormalization on the circle; a renormalization which is not primitive.  
\end{remark}

\begin{example}\label{Yoccoz constr} When $V$ is the circle $\T$, a diffeomorphism $g$ is totally  renormalizable iff it does not fix a point. Indeed in this case, take any point $0\in \T$ and consider the interval  $\Delta =[0,g(0))$. Then we glue the two endpoints of $\Delta$ using $g$ to obtain a circle and we uniformize it to obtain $\T$. This defines a map $H$. For this setting one easily shows that the mapping $g$ is renormalizable. This construction was intensively used by Yoccoz \cite{yoccoz1995petits}. % See also Shil'nikov's parabolic renormalization  \cite{shinikov}.
\end{example}

Let $V$ be a compact manifold   (possibly with corners) and $1\le r\le \infty$. We recall that the {\em support} $\supp f$ of $f\in \Diff^r(V)$ is the closure of the set of points such that $f(x)\neq x$.
\begin{definition}
Let $\Diff^r_0(V)$ be the component of the identity in $\Diff^r(V)$. Let $\Diff^r_c(V)$ be the subset of $\Diff^r_0(V)$ formed by maps isotopic to $\id$ through isotopies $(f_t)_{t\in [0,1]}$ whose  support  $\bigcup_{t\in [0,1]} \supp f_t$  is a compact subset of $V\setminus \partial V$.
\end{definition}
 Observe that when $V$ is boundaryless, it holds $\Diff^r_0(V)=\Diff^r_c(V)$. A natural question is:
\begin{question} \label{question sur la var possible}
For which  manifold $V$, any  map $F\in \Diff^r_c(V)$ is a \emph{total} renormalization of a close to identity map? 
\end{question} 
So far   no example of  such a manifold $V$ was known. In this work  we 
give a full class of  examples:  
%provide evidences to show that a sufficient condition is that $V$ is boundaryless and support a circle action. For instance, a consequence of our main \cref{main} is the following result for $V$ equal to the $n$-torus $\T^n$:
%using  a generalization of the construction in \cref{Yoccoz constr}, we show the following for $V$ equal to the $n$-torus $\T^n$:
 \begin{theo} 
\label{main2} Let $0\le r\le \infty$,  let $M$ be a compact manifold of dimension $\geq 1$ and put ${V:= \T\times M}$. Let $\cN\subset \Diff^r (V)$  be a neighborhood of the identity. Then any $G\in \Diff^r_c(  V)$ is a total renormalization of some $g\in \cN$.% with rescalling domain of the form $ [0,\sigma)\times M$ and such that $g(\{0\}\times M)= \{\sigma\}\times M$.
%Mmap $H: (\theta,y) \in  [0,\sigma)\times M\to (\theta/\sigma, y)\in \T\times M$ for a certain $\sigma>0$.   
\end{theo}
If \cref{main1} implies that every \emph{local} dynamical phenomenon can be found near the identity,  \cref{main2}  implies that every \emph{global} dynamical phenomenon can be found near the identity.  A new improvement brought by the latter result is that the renormalization domain is larger than in all of the previous extensions of Ruelle-Takens theorems: its orbit coincides with the whole domain of the dynamics. In \cref{main}, we will  give a precise formula   defining  the renormalization domain and rescaling map involved in \cref{main2}. This will enable new applications such as the proof of existence of maps preserving smooth SRB near the identity (see \cref{coro thouvenot}) or universal maps whose 
renormalization domains decrease as slow as we want (see \cref{univ}).

 In \cref{cexample}, we show that \cref{main2} is wrong when $\T\times M \approx \T$, hence the set of dimensions of the manifold is optimal. On the other hand, a natural open problem communicated to us by Turaev is:
 \begin{problem}
Show that a dense subset of $\Diff^\omega(\B^n)$ is equal to the renormalization of a close to identity map in $\Diff^\omega(\B^n)$? 
 \end{problem}
 Another natural question is: 
  \begin{question}
Is \cref{main2} correct in the area preserving or symplectic categories?
 \end{question}

\medskip 

 An extension of \cref{main2} regards the $C^r$-families $  f_\cP  = (f_p)_{p\in \cP}$ of maps $f_p\in\Diff^r(V)$ and indexed by a manifold  $\cP$. A family $  (f_p)_{p\in \cP}$ is {\em of class $C^r$} if the following is in    $\Diff^r(V\times \cP)$:
\begin{equation} \label{def widehat}  \widehat{f_\cP}   : = (x,p)\mapsto( f_p(x),p)\; .\end{equation}   
We denote by $\Diff^r(V)_ \cP $ the space of such families endowed with the topology induced by  
$\Diff^r(V\times \cP)$.   Let $\Diff_c^r(V)_\cP$ be the component of the identity in $\Diff^r(V)_\cP$ for homotopies $(f_{p,t})_{(p,t)\in \cP\times[0,1]}\in \Diff^r(V)_ {\cP\times [0,1]} $ whose support $\bigcup_{ (t,p) \in  [0,1] \times \cP} \supp f_{p,t} \times \lbrace p \rbrace $ is a compact subset of  the interior of $V \times \cP $. Observe that $ \widehat{f_\cP}  \in \Diff^r_c ( V \times \cP )  $.

%the set of families $(f_p)_{p\in \cP}$ {\em compactly homotopic to the identity family:} 
%  there exists a $C^r$-family $(f_{p,t})_{(p,t)\in \cP\times[0,1]}$ such that 
%  $(f_{p,0})_{p\in \cP}=(id)_{p\in \cP}$, $(f_{p,1})_{p\in \cP}=(f_p)_{p\in \cP}$ and
%  the union of the supports of the  $f_{p,t}$ is in the interior of $V$.
 
%  of the form: 
%  \begin{equation} f_{\cP,t}:= (x,p,t)\mapsto( f_{p,t}(x),p), \quad \forall t\in [0,1],\end{equation} 
 
%  $f_\cP$ is isotopic to the identity via a homotopy $(f_{\cP,t})_{t\in [0,1]}$ of the form:
% \begin{equation} f_{\cP,t}:= (x,p)\mapsto( f_{p,t}(x),p), \quad \forall t\in [0,1],\end{equation} 
% and such that the union of the supports of the  $f_{\cP,t}$ is in the interior of $V\times \cP$.  
%\setcounter{theo}{2}
\begin{theoprime}\label{main2prime}  
Let $0\le r\le \infty$, let $M$ and $\cP$  be  compact manifolds  of dim $\geq 1$  and set  ${V:= \T\times M}$. Let $\cN\subset   \Diff^r (V)_\cP$ be a neighborhood of $(id)_{p\in \cP}$. 

Then for any ${(G_p)_{p\in \cP}\in \Diff^r_c( V)_\cP}$, there exist 
$(g_p)_{p\in \cP}\in \cN$ and  a rescaling map  (independent of $p\in \cP$)
 of a  total renormalization domain    which  renormalizes each $g_p$ to $G_p$.
\end{theoprime}
 This theorem implies that any  bifurcation in $\Diff^r_c (V)$ occurs at small unfolding of the identity.  
In \cref{lieu thm D}, we will state the main general   \cref{main} which implies 
Theorems \ref{main1} and \ref{main2} and also its parametric counterpart \cref{mainprime} which implies \cref{main2prime}.  In \cref{applications} we will give several applications of them. 
Now let us discuss the optimality of \cref{main2}. 
\begin{proposition}\label{cexample}  When $r\ge 2$,  \cref{main2} is wrong if $V \approx \T$,   i.e. when it is isomorphic to the circle as a smooth manifold   (and so $\dim M=0$). 
\end{proposition} 
\begin{proof}
{  Indeed  if $F$ is a renormalization of a close to  identity map $f$ for a renormalization domain $\Delta$, then  we have necessarily $|\tau(\theta)-\tau(\theta')|\le 1$ for any $\theta,\theta'\in \Delta$. Let $N:= \min\{\tau(\theta),\tau(\theta')\}$.  We compute the derivative of the $N-1$ first iterates $(\theta_i)_i$ and   $(\theta'_i)_i$ of $\theta$ and $\theta'$:
\begin{equation*} \log \left|\frac {D_\theta F}{D_{\theta'} F}\right|= \log \left|\frac{D_\theta f^N}{D_{\theta'} f^N}\right| +o(1)\quad \text{when } f\to id \end{equation*} 
\begin{equation*} =\sum_{i=0}^{N-1} \log \left|\frac{D_{\theta_i} f}{D_{\theta'_i} f}\right|+o(1)\le \|\log | D f|\|_{C^{1}}\cdot \sum_{i=0}^{N-1}  |\theta_i-\theta_i'|+o(1)=o(1),\end{equation*} 
where the latter inequality uses that $ \|\log | D f|\|_{C^{1}}$ is small while the segments  $[\theta_i,\theta_i']$ are disjoint and so the union of their length is at most $1$. Hence this proves that the derivative of $F$ is constant and so that $F$ must be a rotation.  }
\end{proof}

  Also we cannot change $\Diff_c^r ( \T\times M)$ by $\Diff_0^r ( \T\times M)$ in \cref{main2}. Indeed the latter proposition applied to the boundary of $[0,1]$ implies immediately:
\begin{coro}\label{cexample1} 
There are $G\in \Diff_0^\infty ( \T\times [0,1])$ which are not total renormalization of $C^2$-close to identity map.
\end{coro}
Yet in view of \cref{question sur la var possible}, \cref{main2} seems to be generalizable for a manifold $V$ on which $\T$ acts properly discontinuously without any fixed point.\\

\thanks{ \em We are grateful to the referee for their thoughtful corrections and suggestions.}

 \subsection{Applications and open problems}
 \label{applications}
\paragraph{Smooth SRB near the identity}
In \cref{subsec plugins} we will state \cref{main} which will imply together with  Katok's theorem \cite{katok1979bernoulli}, an answer to an open question of Thouvenot:
  \begin{corollary} \label{coro thouvenot}
In any neighborhood $\cN$ of $\id\in \Diff^\infty(\T^2)$ there is a map $g\in \cN$ which leaves invariant an ergodic smooth volume form and displays positive Lyapunov exponent at Lebesgue a.e. point.\end{corollary}
This corollary will be proved in \cref{proof coro thouvenot} from \cref{coro thouvenot generalise}, and generalized to  higher dimension using   \cite{dolgopyat2002every}.
\paragraph{Universal mappings}
A map $f \in \Diff^r( V ) $ is said to be {\em universal} if there exists a dense subset of $ \Diff_0^r(V) $ such that each of its elements is a renormalization of $f$. Bonatti and D\'iaz in \cite{bonatti2003maximal}  have shown that universal  maps are locally $C^1$-generic on $\B^3$.  Turaev in \cite{TUR15} has shown that
 universal  maps are locally $C^\infty$-generic on $\B^2$. Yet mathematicians wondered whether we can ``see" the universality of such mapping. While there are infinitely many renormalization domains, the above proofs lead to a very small volume for the union of  their orbits. 

\begin{corollary} \label{univ}
Let $1\le r\le \infty$ and  $n\ge 2$. For any sequence $(F_i)_{i\ge 0}$ of maps $F_i\in  \Diff_0^r(\T \times [ 0,1] )$ and any sequence of positive numbers $(\ell_i)_{i\ge 0} $ s.t. $\sum \ell_i< 1 $, there exists a $C^r$-arbitrarily close to identity map $f \in \Diff_0^r(\T \times [ 0,1] )$ which displays a family of  renormalization domains $(\Delta_i)_{i\ge 0}$ such that:
\begin{enumerate}
\item a renormalization of $f$ associated to $\Delta_i$ is $F_i$ for every $i\ge 0$, 
\item the orbit $\tilde \Delta_i := \bigcup_{n\ge 0} f^n(\Delta_i)$ has volume equal to $\ell_i$,
\item the  sets $\tilde \Delta_i $ and $\tilde \Delta_j $ are disjoint for $i\neq j$. 
\end{enumerate}
%there exists $C^r$-universal maps whose renormalizations   realize each elements of $D$. 
\end{corollary}
This corollary will be proved in \cref{proof univ}.

The proof of the main theorem is constructive and it seems to us that, in the case where $M$ is  boundaryless,  the map $g$ of \cref{main2}  depends smoothly on $G$ in a neighborhood of the identity. This leads us to propose:
\begin{conjecture}
For every compact boundaryless manifold $M$ of dimension $\ge 1$, there exists  a neighborhood $\cN_0$ of $id \in \Diff^\infty(\T\times M)$ such that for every neighborhood $\cN$ of ${id \in \Diff^\infty(\T\times M)}$, there is a smooth (tame) injective map ${\cal I}: G\in \cN_0\mapsto g\in \cN$ such that $G$ is a total renormalization of $g={\cI}(G)$  for every $G\in \cN_0$. 
\end{conjecture}
Roughly speaking,  this conjecture asserts that modulo total renormalization, a   fixed neighborhood  $\cN_0$ of $id \in \Diff^\infty(\T\times M)$ can be smoothly embedded into any smaller neighborhood. This defines infinitely many inverse branches of the renormalization operator with image converging to the identity.

\subsection{Sketch of proof}
 \paragraph{Plugins and pluggable dynamics:}
 The framework of the proof of the main theorem relies on a new object called \emph{plugin} and the notion of \emph{pluggable map}.  A plugin is a renormalizable map of a special form, so that it has a canonical renormalization called its \emph{output}. See Def. \ref{def.plugin} and \ref{def.output} below and \cref{fig1,fig2}. We will say that a map is \emph{pluggable} if it is the output of an arbitrarily close to  identity plugin and likewise for its inverse. In particular a pluggable map is a total renormalization of a close to identity map. 
 Most of this work will be dedicated to show   \cref{mainplug} stating that:
 \[ \textit{any map of $\Diff^\infty_c(\T\times M)$ is pluggable.} \]
 The finite regularity  counterpart of \cref{mainplug}  is stated as \cref{main} in \cref{subsec plugins} and will be deduced from \cref{mainplug} in \cref{lieu cas CR}.
 We will deduce Theorems \ref{main1} and \ref{main2} and \cref{coro thouvenot,univ} from \cref{main} in \cref{proof ABC}.

In \cref{subsec top ex} we  precise the topologies of the involved spaces. Also we will show that the following group is formed by pluggable maps, see \cref{twist cP}:

\begin{equation*} \G_1:=  \{ (\theta, y) \in \T \times M    \mapsto  (\theta+ 
 \nu(y) , y ) \in \T \times M  : \nu \in C^\infty_c(M, \T)\}\; .\end{equation*}
\paragraph{Topological group structure on $\P$.}
 It is easier to work with pluggable maps rather than directly the set of  total renormalizations of  close to identity maps. Indeed, we will show  in \cref{group} that the set $\P$ of pluggable maps endowed with the composition rule $\circ$ is a group. To prove this,  we will define in \cref{group section} a binary operation $\star$ on compatible plugins $g_1,g_2$ such that the 
 output of $g_1\star g_2$ is the composition of the outputs of $g_1$ and $g_2$.  See \cref{conca plug}. 
In \cref{section:Dil} we will show that the following group is formed by pluggable maps, see \cref{Dil}:
 \begin{equation*} \G_2:= \left\{ (\theta, y)\in \T\times M   \mapsto (\theta,F(y)) \in \T \times M  : 
F \in \Diff_c^\infty(  M)\right\} \; . \end{equation*} 
Then in \cref{subsec close} we will show that $\P$ is closed in~$\Diff^\infty_c( \T \times M )$, see \cref{c.closure}. 
To prove this, we will construct plugins whose dynamics enlarge an iterate of the renormalization domain and will perform a perturbation therein. Note that the elements of $ \Diff_c^\infty ( \T \times M) $ generated by compositions of elements of $ \G_1 $ and $ \G_2 $ have constant derivatives w.r.t. $\theta$.  Consequently we will need to construct other diffeomorphisms in $ \P $ to prove \cref{mainplug}. 
To do so we will consider the space: 
\[  \diff_c^\infty ( \T \times M)  \text{ of compactly supported vector fields on  } \T \times M  \; ,    \]
and study the space $ \spp $ of vector fields whose flow is pluggable: 
\begin{equation*} \spp:= \{X\in \diff_c^\infty(\T\times M): \Fl^t_X\in \P,\quad \forall t\in \R\}\; ,  \end{equation*} 
where $\Fl^t_X$ denotes the flow of $X$ at time $t$. 
Using that $\P$ is a closed subgroup, we will deduce in \cref{subsec debut vector} that $\spp$ is a closed sub-Lie algebra of $\diff_c^\infty ( \T \times M)$, see \cref{p.sppclose}. 

 Note that the following Lie algebras are formed by fields whose flows are in $\G_1$ or $\G_2$:   
\begin{equation*}  \sg_1  = \lbrace  X_1 : ( \theta , y ) \in \T \times M \mapsto ( v (y) , 0 ) \in \R \times T M  \ : \  v \in C^\infty ( M , \R )   \rbrace   \subset \spp\; ,\end{equation*}  
\begin{equation*} \sg_2 = \lbrace X_2 : ( \theta , y ) \in \T \times M \mapsto (0  ,  f(y)  ) \in \R \times T M  \ : \  f  \in \diff_c^\infty ( M    )    \rbrace\subset \spp \; .\end{equation*} 

\paragraph{Construction of pluggable flows.}
In \cref{subsec proof of main}, using the connectedness of ${\Diff_c^\infty ( \T \times M)}$ and that  $\P$ is a closed group,  we will show that, to prove \cref{mainplug}, it suffices to show:
\begin{equation*} \spp = \diff_c^\infty ( \T \times M)  \; . \end{equation*}  
This equality will be stated in  \cref{p.mainspp}.  Its proof  relies on two phenomena. 
The first one is stated  in \cref{p.adjunctionstability}  as:
\[  \lbrace Y \in \diff_c^\infty ( \T \times M) : \exists X \in \spp \text{ such that }  Y = [ X ,Y ] \rbrace \subset \spp \; , \]
where $ [ \cdot , \cdot ] $ denotes the Lie brackets.
This will be proved in  \cref{Eigen} by noting that for any $Y = [ X ,Y ] $, it holds
$\Fl_{ Y}^t= \Fl_X^{s}\circ \Fl_{ Y}^{t\cdot e^{-s}} \circ\Fl_X^{-s} $. Then by using the same technique as for the proof of the closedness of $ \P $, 
we will show that arbitrarily close to identity there exists a plugin with output $\Fl_{e^{-s} Y}^t$ for any $s$ large enough. 
We will conclude the proof by doing  two $\star$ products of the latter with plugins with outputs $ \Fl_X^{s}$ and $ \Fl_X^{-s}$.

The second phenomenon, stated in \cref{p.pre_deceig}, is that for any vector field $T \in\diff_c^\infty ( M    )   $, there exist finite families $(X_i)_i,(Y_i)_i,(Z_i)_i$ of vector fields in $  \diff_c^\infty ( M    )    $ such that: 
\begin{equation*}  T=\sum_i [Y_i,Z_i]\qand Y_i=[X_i,Y_i]\; .\end{equation*} 
From this we will deduce the same statement for vector fields in $\sg_2$.  To show \cref{p.pre_deceig}, we will remark that when   $M=\R$, for the vector fields  $X : y \in \R \mapsto -y $ and $Y = 1 $, it holds  $ Y = [ X, Y] $ and for any $T\in \diff^\infty_c(\R)$,  with  $ Z = \int_{ - \infty}^y T(t) dt  $, it also holds $ T = [ Y, Z ] $. Then we will deduce a compactly supported and parametric version of this property which will enable us to prove 
 \cref{p.pre_deceig} in the case $M=\R^n$. Finally, we will use a partition of unity to deduce    the proposition for any manifold.
 
These phenomena will enable us to prove \cref{p.hatg2} stating that for any $T\in \sg_2$ and $\phi\in C_c^\infty(\T\times M)$ depending only on $\theta$, the field $\phi\cdot T$ is in $\spp$. Indeed,  by the second phenomenon, there are $X_i,Y_i,Z_i\in \sg_2$ s.t.:
 \begin{equation*}  \phi\cdot T=\sum_i  \phi\cdot  [Y_i,Z_i]= \sum_i     [\phi\cdot  Y_i,Z_i]\qand Y_i=[X_i,Y_i]\; .\end{equation*}  
 Also  a simple computation shows that  $ [\phi\cdot  Y_i,X_i]=\phi\cdot Y_i$. As $X_i$ is in $\sg_2\subset \spp$, we deduce by the first phenomenon that $\phi\cdot Y_i$ is in $\spp$. This gives that  $ \phi\cdot T\in \spp$ as stated in  \cref{p.hatg2}.
 In \cref{subsec proof of main}, we will use Fourier decomposition Theorem and the closedness of $\spp$ to deduce 
\cref{p.hatg2} stating that any vector field of the form $(0, Y(\theta,y))$ is in $\spp$. Finally using this with a Lie bracket with an element of $\sg_1$ will enable us to obtain that any vector field has a pluggable flow ($\spp=\diff^\infty_c(\T\times M)$) as stated by \cref{p.mainspp}. 
\paragraph{Parametric counterparts.}
At the end of each subsection, we will prove a parametric generalization of the aforementioned statements. This will enable us to show the parametric counterpart  \cref{mainplug cP} of  \cref{mainplug}. It will imply the 
parametric counterparts  \cref{mainprime} of  \cref{main} and \cref{main2prime} of  \cref{main2}.  

 \section{Plugins and Pluggable dynamics} \label{sec plugins and pluggable dynamics} 
 \subsection{Plugins} \label{subsec plugins}
 \begin{center} \emph{For the rest of this article, we fix $1\le r\le \infty$ and  compact connected manifolds $M$ and $\cP$ of $\dim\ge 1$. }\end{center} 
 For $\sigma>0$, define the rotation: 
\begin{equation*} R_\sigma: (\theta,y)\in \T \times M\to (\theta+\sigma ,y)  \in \T\times M \; . \end{equation*} 
We are now ready to introduce: \label{lieu plugin}
\begin{definition} \label{def.plugin}
 A {\em plugin}  with {\em step} $\sigma\in   \{2^{-k} : k\ge 1 \} $ is a map $g\in  \Diff^r(\T\times M)$ satisfying the following assertions:
\begin{enumerate}[$(i)$]
\item     $g$ restricted to $\Delta_\sigma:=[0,\sigma) \times M$ is equal to $R_\sigma$,% translation $(x,y)\mapsto (x+\sigma, y)$,
\item  the first  return time in $ \Delta_\sigma$ of $g $  is a well defined and bounded  function  $\tau: \Delta_\sigma\to \N^*$,
\item the union of the iterates $\bigcup_{k\ge 0} g^k(\Delta_\sigma)$ equals $\T\times M$.
\end{enumerate}
\end{definition}
\begin{remark}
One can show by compactness of $M$ that, under condition $ ( i) $, condition $(ii)$ is equivalent to $ ( iii ) $ and that in $ (ii ) $ the return time is necessarily bounded.
\end{remark}
\begin{figure}[H]
\begin{center}{\footnotesize \hskip -2mm\raisebox{8ex}{$\T\times M\colon$} {\def\svgwidth{68pt}          %\'echelle du dessin
 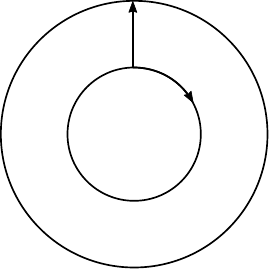}%\pause
\hskip 7mm \raisebox{-0ex}{\def\svgwidth{82pt}          %\'echelle du dessin
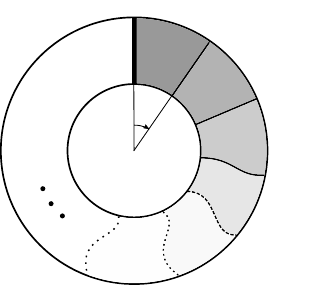}%\pause
\hskip 2mm \raisebox{-0.3ex}{\def\svgwidth{68pt}          %\'echelle du dessin
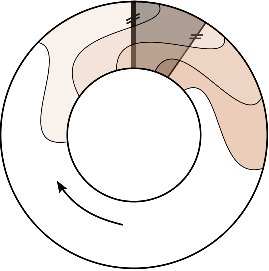}}\end{center}
\caption{Plugin $g$ of step $\sigma$.}\label{fig1}
\end{figure}
Let $H_\sigma:= (\theta,y)\in \Delta_\sigma\mapsto (\theta/\sigma,y)\in \T\times M$.  It is a bijective local diffeomorphism.
\begin{figure}[H]
\begin{center}{\footnotesize\def\svgwidth{120pt} \raisebox{-2.4ex}{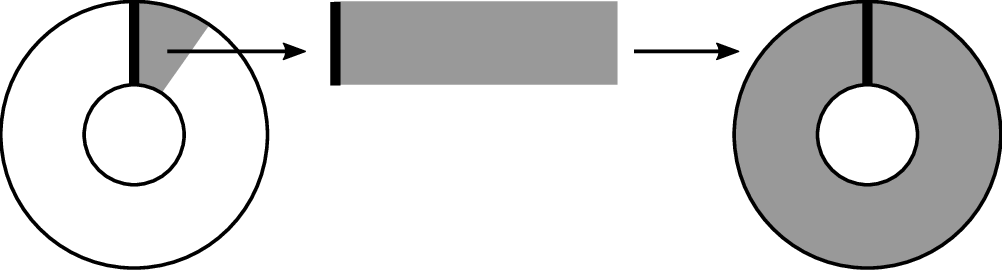}
}\end{center}
\caption{Rescaling map $H_\sigma\colon \Delta_\sigma\to \T\times M$.}\label{fig2}
\end{figure}

\begin{definition}\label{def.output}
The {\em output} of a plugin $g$ of step $\sigma$ is the following rescaling of the first return map $g^\tau \colon\Delta_\sigma\to \Delta_\sigma$: \begin{equation*} G:=H_\sigma\circ g^\tau\circ H^{-1}_\sigma\colon \T\times M\to \T\times M.\end{equation*} 
\end{definition}
 \begin{example}\label{exem id0}
For every $k\ge 0$, the map 
$g_{k}: (\theta, y)\mapsto (\theta+2^{-k}, y)$ is a plugin of step $2^{-k}$, iteration $2^k$ and output the identity.
\end{example}
Actually, we can show that the output of a plugin is always smooth: 
\begin{proposition}\label{p.diffeocontinuity}
Let $1\le r\le \infty$. The output of a plugin $g\in \Diff^r(\T\times M)$ is in $\Diff^r(\T\times M)$ and depends continuously on $g$. In particular, the output is a total renormalization of the plugin.
\end{proposition}
This proposition is a consequence of a classical, yet beautiful, argument which will be recalled in \cref{lieu preuve p.diffeocontinuity}. We are now ready to state the general result:
\label{lieu thm D}
\begin{theo}[Main]
\label{main} Let $1\le r\le \infty$ and a compact manifold $M$ of dimension $\geq 1$, and let $\cN\subset \Diff^r (\T\times M)$  be a neighborhood of the identity. Then any $G\in \Diff^r_c(  \T\times M)$ is the output of some plugin $g_G\in \cN$.   
\end{theo}

And here is its parametric counterpart:
\begin{theoprime}\label{mainprime}
Let $1\le r\le \infty$ and a compact manifold $M$ of dimension $\ge 1$,
fix a compact manifold $\cP$ and a compactly supported family $(G_p)_{p\in \cP}$ in $ \Diff_c^r ( \T \times M ) _\cP$. 
Let ${\cN\subset   \Diff^r (\T \times M)_\cP}$ be a neighborhood of $(id)_{p\in \cP}$.
Then there exists a $C^r$-family of plugins $( g_p )_{ p \in \cP} \in \cN$ such that 
$G_p$ is the output of $g_p$ for every $p\in \cP$. 
\end{theoprime}

\subsection{Proof of the corollaries of the main  \cref{main}}\label{proof ABC} \label{subsec proof of corollaries and main}
Observe that \cref{main} implies immediately \cref{main2} and that \cref{mainprime} implies immediately \cref{main2prime}. In this subsection we show that 
\cref{main} implies furthermore \cref{main1} and   \cref{coro thouvenot,univ}. To this end, the following will be useful:

\begin{fact} \label{fact inv}
If the output of a plugin restricted to $\{0\}\times M$ is the identity, then the return time $\tau $ of the plugin is constant on the renormalization domain  $\Delta$. 
\end{fact}
\begin{proof}As $f$ preserves the orientation it suffices to show that $\tau $ is constant on the interior of $\Delta$. 
As $\tau $ is integer valued and $M$ is connected, it suffices to show that $ \tau $ is continuous on $int\, \Delta$ to conclude the proof.
We start by showing that  $ \tau $ is lower semi-continuous on $\Delta $. Suppose that it is not the case. Then there exist $N\ge 2$, a point $x\in \Delta$ and a point $ x' \in \Delta $  arbitrarily close to $x$  such that  $ N = \tau( x') < \tau (x) $. 
By continuity of   $ g $, it holds $ g^N ( x) \in \cl( \Delta) $. By assumption, we have $ g^N(x) \notin \Delta $. Thus $ g^N(x) \in \lbrace \sigma \rbrace \times M $ and since $g$ is the translation by $ \sigma $ on $ \Delta $ it holds then that $g^{N-1} ( x) \in \Delta$ and consequently $ \tau(x) < N $, which contradicts the assumption. 

We now show that $ \tau $ is upper semi-continuous on $int\, \Delta$.
Consider $ x \in int\,\Delta$. As $g^\tau$ is a bijection of $\Delta$ which leaves invariant $\{0\}\times M$, it comes that $g^\tau(x)$ is in the interior of $\Delta$. 
%If $g^{\tau(x)} (x)$ is in the interior of $\Delta$, 
Then for every $x'$ close to $x$, the iterate  $g^{\tau(x)} (x')$ belongs to $\Delta$ and so $\tau(x')\le \tau(x)$. 
%Otherwise $g^{\tau(x)} (x)$ is in $\{0\}\times M$. As $g^\tau$ is a bijection of $\Delta$ which leaves invariant $\{0\}\times M$, it comes that $x\in \{0\}\times M$. 
%Also  $\tau$ is constant on the connected set $ \lbrace 0 \rbrace \times M $ ; denote $ n $ its value on $ \lbrace 0 \rbrace \times M $.
%Thus,  to conclude the proof, it suffices to show  the existence of a neighborhood of $\partial \Delta$ in $\Delta$ which is sent into $\Delta$ by $g^{n}$.
%
%For the sake of contradiction assume the existence of a sequence $x_k\in int\,  \Delta \to x$ which is sent  outside of $\Delta$ by $g^n$. Then  $(x_k)_k$ must approach $x$ by the right hand side of $\{0\}\times M\ni x$, while $(g^n(x_k))_k$ approaches $g^n(x)$ by the left hand side of $\{0\}\times M\ni g^n(x)$.  Up to extracting a subsequence, there is a curve $\gamma: \T \to \T\times M$ which contains all the points $x_n$ so that $\gamma(0)=x$ and which projects by $p_\theta: \T\times M\to \T$ to a degree one map $p_\theta\circ \gamma$. As $g^n$ is homotopic to the identity, the map $p_\theta\circ g^n\circ \gamma$  has also degree one. 
%As $  g^n $ preserves $ \lbrace 0 \rbrace \times M $, the curve  $\theta \in \T \mapsto g^n ( \theta , y) $ intersects $ \lbrace 0 \rbrace \times M $ exactly once; this must be at $g^n\circ \gamma(0)$. So nearby $0$, the map   $p_\theta\circ g^n\circ \gamma$ must be non-decreasing. This  contradicts that $(g^n(x_k))_k$ approaches $g^n(x)$ by the left hand side of $\{0\}\times M\ni g^n(x)$.
\end{proof}
\begin{proof}[Proof that \cref{main} implies \cref{main1} \label{proof main1}] 
Let $G$ be in $\Diff^r(  \bB^n)= \Diff^r_0(  \bB^n)$. 
We observe\footnote{ \label{fn 1}By connectedness, there exists 
an isotopy $(h_t)_{t\in [0,1]}$ between   $G|\partial \B^n$ and 
$id_{\partial \B^n}$ that can be chosen $C^r$-smooth, made of diffeomorphisms and  and flat at the endpoints. Note that we can extend $G$ on $2\cdot \B^n\setminus \B^n$ by $ h_{\|z\|-1} \left(\frac{z}{\|z\|} \right)$ to construct and element of $\Diff^r_0(2\cdot \bB^n)$.} that $G$ is the restriction to $\B_n$  of a diffeomorphism $\bar G$ in $\Diff^r_c(  2\cdot \bB^n)$. As $2 \cdot id$ conjugates  $\bar G$ to a map in  $\Diff^r_c(   \bB^n)$,
without any loss of generality we can assume that $G$ belongs to $\Diff^r_c(   \bB^n)$. 

Let $M:=\bB^{n-1}$ and embed $i: \bB^n\hookrightarrow \T\times M$ so that the embedded ball $i(\B^n)$  does not meet $\{0\}\times M$ nor the boundary of $\T\times M$. Extend then $G$ by $\id$ to a diffeomorphism $\tilde{G}\in \Diff_0^r(\T\times M)$. By \cref{main}, $\tilde{G}$ is the output of a plugin $\tilde{g}\in  \Diff^r(\T\times M)$ arbitrarily close to $\id$. Since $\tilde G $ leaves $ \lbrace 0 \rbrace \times M  $ invariant it holds by \cref{fact inv} that the first return time $\tau$ of the plugin  $ \tilde g  $ is a constant $N$. In particular, the restriction $\tilde g^N| i(\B^n)$ is conjugate to $G$.    

\begin{figure}[H]
\begin{center}
{\footnotesize \begin{equation*} \tilde{G}\colon\!\!\raisebox{-6ex}{\def\svgwidth{55pt}           %\'echelle du dessin
 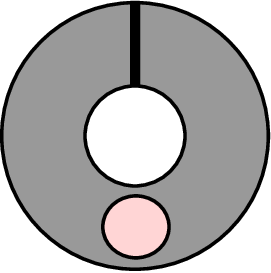}%\pause 
 \mbox{ is}\underset{\mbox{$\tilde g$ close to $\id$}}{\mbox{ the output of } \tilde g}\colon \raisebox{-6ex}{\def\svgwidth{55pt}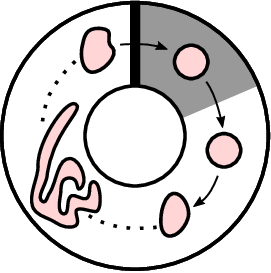}%\pause 
 \lhook\joinrel\xrightarrow{\begin{array}{c}\mbox{embed}\\\mbox{in }\B^d\end{array}} \raisebox{-6ex}{\def\svgwidth{55pt}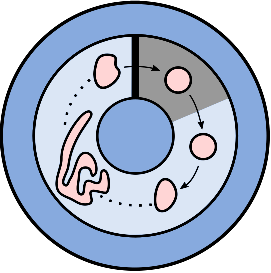}%\pause
 \begin{array}{c}\mbox{Extend to }\\g\in \Diff(\B^d)\end{array} \end{equation*} }\end{center}
\caption{Proof of \cref{main1}.}\label{fig3}
\end{figure}
To conclude it suffices to embed $\T\times M$ into the interior of $\B^n$ and extend $\tilde g$ then  the diffeomorphism $\tilde{g}$ to a diffeomorphism $g$ of $\B^n$ that is close to $\id$. \end{proof}
 
The following enables us to localize near the identity the existence of some ergodic properties:
\begin{proposition}\label{coro thouvenot generalise}
Let $g$ be a plugin of step $\sigma$ and $G$ its output. Then there is a canonical bijection $\mu_g\to \mu_G$ between  $G$-invariant probability  measures and  $g$-invariant probability measures:
\[ \mu_g\mapsto \mu_G:= H_{\sigma*} \frac{\mu_g|\Delta}{\mu_g(\Delta)} \qand \mu_G\mapsto \mu_g:= \sum_{N\ge  0} g^{*N} H_\sigma ^*\mu_G |\{\tau > N\}\; .\]
 Moreover:
 \begin{enumerate}
 \item  $ \mu_g$ is ergodic  iff $\mu_G $ is ergodic,
 \item $ \mu_g$ is hyperbolic   iff $\mu_G $ is hyperbolic. 
 \item  $\mu_g$ is a smooth  volume form iff   $\mu_G $ is a smooth volume form,
\end{enumerate}
\end{proposition}
\begin{proof}
We recall that $H_\sigma$ conjugates $G$ with the first return time $g^\tau$ of  $g$ in $\Delta$. Then it is sufficient to prove the analogous properties for the following  
classical  bijection between   $g^\tau$-invariant probability  measures and  $g$-invariant probability measures: 
\[ \mu_{g^\tau}\mapsto \mu_g:= \sum_{N\ge  0} g^{*N} \mu_{g^\tau} |\{\tau > N\}\qand\mu_g\mapsto \mu_{g^\tau}:=  \frac{\mu_g|\Delta}{\mu_g(\Delta)} \; .  \]
Indeed it is well known that $\mu_{g^\tau}$ is ergodic iff $\mu_g$ is ergodic (see for instance \cite[Lemma 2.43]{einsiedler2013ergodic}), this implies 1.  
Also it well known that the Lyapunov exponent of $\mu_G$ are equal to the mean of $\tau$ times the Lyapunov exponent of $\mu_g$, from which we deduce 2. 

Let us prove 3. If $\mu_{G}$ is smooth, then $\mu_g|\Delta= \mu_{g^\tau}$ is smooth on $\Delta$. It follows that ${\mu_g|\Delta\cup g(\Delta)= \mu_{g^\tau}+g^*\mu_{g^\tau}}$ is smooth on $\Delta\cup g(\Delta)$.
Now observe that for every $x\in \T\times M$ there exists a neighborhood $U$ and $N\ge 0$ such that $g^{-N} (U)\subset \Delta\cup g(\Delta)$. Then by invariance, the density of $\mu_g$ on $U$ is as smooth as $\mu_g|\Delta\cup g(\Delta)$.  Conversely, if $\mu_g$ is smooth, then  $\mu_{g^\tau}:=  \frac{\mu_g|\Delta}{\mu_g(\Delta)}$ must be smooth. Since $g$ is a plugin, the density $\mu_g$ at the left hand side of $\Delta$  is equal to the translation by  $(\sigma, 0)$ of the density of $\mu_g$ at the right hand side of $cl(\Delta)$. Thus $\mu_{g^\tau}$ is pushed forward by $H_\sigma$ to a smooth volume form $\mu_G$ on $\T\times M$. 
\end{proof}
\begin{proof}[Proof of \cref{coro thouvenot}]
\label{proof coro thouvenot}
Katok in  \cite{katok1979bernoulli}[Thm B] showed the existence of Bernoulli diffeomorphisms  homotopic to the identity and preserving a given smooth measure on any surface. Hence there exists   $G\in \Diff^\infty_{0} (\T^2)$ preserving a smooth  Bernouilli volume form $\mu_G$. Equivalently  by the Pesin Theorem the measure $\mu_G$ is a volume form which is ergodic and hyperbolic.  
By \cref{main}, arbitrarily close to the identity, there exists $g$ whose output is $G$.  Then by  \cref{coro thouvenot generalise}, the volume form $\mu_G$ induces an  ergodic  and hyperbolic volume form $\mu_g$ for $g$. 
\end{proof}

We now proceed to: 
\begin{proof}[Proof of \cref{univ}] \label{proof univ}
We will construct a sequence $ (\Delta_i)_{i \ge 1 } $ of domains with disjoint closures and of the form $ \Delta_i := [ 0, \sigma_i) \times (s_i, s_i + l_i)   \subset \T \times [ 0,1] $ with $  0 < \sigma_i < 1  $ and $ 0 < s_i < 1-l_i $. Then we will be able to construct a map $f$ that renormalizes to $F_i$ on each $ \Delta_i$ and such that the orbit of the renormalization domain $ \Delta_i$ is exactly $ \T \times ( s_i, s_i + l_i)$. Moreover we will be able to take   $f$ arbitrarily   close to identity. 

To do so, consider a sequence $ ( l'_i)_{i \ge 0} $ such that $ \sum_i l'_i < 1 $ and $ l_i' > l_i $ for any $i \ge 0$. 
For $i \ge 0 $ denote    $ s_i  := \sum_{j < i } l'_j$ and $ B_i := \T \times [ s_i, s_i + l_i ] $. Note that $ B_i$ has volume $l_i$ and that $ \cl ( B_i )   \cap \cl (B_j) = \emptyset $ for any $ i \ne j \ge 0$. Let us define the map $ \psi_i : ( \theta ,y) \in \T \times [0,1]  \mapsto ( \theta , s_i + l_i \cdot y )  $ that sends $ \T \times [ 0,1]  $ to  $ \cl ( B_i )  $.  Now by \cref{main}, for any $ i \ge 0$, there exists a plugin $ \tilde f_i :  \T \times  [ 0,1 ] \to  \T \times  [ 0,1 ] $ arbitrarily close to 
identity with renormalization domain with output $F_i$. Let  step $ \sigma_i $ be its step. The map  $  f _i := \psi_i  \circ \tilde f _i \circ \psi_i^{-1}  : B_i  \to B_i    $ is close to identity. Also a renormalization of $f_i$ associated to $ \Delta_i := [0, \sigma_i)  \times [ s_i, s_i + l_i ] $ is $F_i$.
Since the supports $ \supp f_i =  \cl ( B_j )   $  are disjoint and each $f_i$ is close to identity, there exists a close to identity map $f : \T \times [0,1] \to  \T \times [ 0,1]$ that coincides with  $f_i$ on $ B_i$. Such a map verifies the conditions of the \cref{univ}. 
 \end{proof} 

\subsection{Pluggable dynamics} \label{subsec top ex}
We will first work in the $C^\infty$-topology. Indeed, the proof relies on the fact that $\Diff^\infty(V)$ is a Fr\'echet Lie group, which is not the case of $\Diff^r(V)$ for $r<\infty$. The $C^r$ case will be deduced  from the $C^\infty$ case in \cref{lieu proof main}.
\subsubsection{Topologies on spaces of smooth maps and parameter families}\label{sec:def f cp}
We endow $V$ with a Riemannian metric.  Let  $\diff^\infty(V)$ be the space of  smooth vector fields  on $V$  that are tangent to the boundary (if any). 
We  endow $\Diff^\infty( V)$ and 
 the space $\diff^\infty(V)$  with the following distances: 
 \begin{equation*} 
d_{C^\infty}(f,g)= \max_{x\in V }\sum_{k\ge 1} 2^{-k} \min (1, \|D_x^k f - D^k_x  g \|)  \end{equation*} 
and
\begin{equation*} 
d_{C^\infty}(X,Y)= \max_{x\in V }\sum_{k\ge 1} 2^{-k} \min (1, \|D_x^k X-D^k_x  Y\|)\; . \end{equation*} 
For these distances $ \Diff^\infty(   V)$ and  $ \diff^\infty(V)$ are  complete. Actually $ \Diff^\infty(   V)$ is a Lie group with algebra the Fr\'echet space $\diff^\infty(V)$. 
We endow the connected component $\Diff^\infty_0(V)$  of the identity with the topology induced by $\Diff^\infty(V)$.  

On the other hand, we endow the space $\Diff^\infty_c(   V)$ and the space of compactly supported smooth vector fields $\diff^\infty_c(V)$ with the finer {\em Whitney topology}. A basis of open sets of these respective topologies is:  
\begin{equation*} U_{\eta, f,m} := \{g\in \Diff^\infty_c(   V):  \| D_x^k f - D^k_x  g \| \le \eta(x), \forall k\le m\} \end{equation*}  
and
\begin{equation*} 
U_{\eta, X,m} := \{Y\in \diff^\infty_c(   V):  \|D_x^k X- D^k_x  Y\|\le \eta(x), \forall k\le m\}\; , \end{equation*}  
among $m\in \N$, $f\in \Diff^\infty_c(   V)$, $X\in \diff^\infty_c(   V)$,  and %positive
continuous functions ${\eta: V\setminus \partial V \to (0,\infty)}$.   A well known theorem asserts that  $f_n\to f$ in  $\Diff^\infty_c(   V)$ if and only if   there exists a compact subset $K\subset   V\setminus \partial V$  such that the supports of $f$ and $f_n$ are included in $K$ for  $n$ and 
for every $r\ge 1$, we have $f_n\to f$ in the uniform $C^r$-topology when $ n \to \infty$.  The analogous property holds true for  $\diff^\infty_c(   V)$. 

We endow $ \Diff^r( V)_\cP$ and $ \Diff_c^r( V)_{\cP} $ with the topologies induced by $ \Diff^r( V \times \cP ) $ and $ \Diff_c^r( V \times \cP ) $ with the inclusions:
\begin{equation*} \Diff^r( V)_{\cP} \hookrightarrow  \Diff^r( V \times \cP )  \qand \Diff_c^r( V)_{\cP} \hookrightarrow  \Diff_c^r( V \times \cP ) 
\end{equation*} 
via $ f_\cP = ( f_p) _{p \in \cP} \mapsto  \widehat {f_\cP}$ where:
\begin{equation*} \widehat {f_\cP}   := ( v , p ) \in V \times \cP  \mapsto ( f_p ( v) , p)\; .\end{equation*}
The spaces $ \Diff^r ( V)_\cP$ and $ \Diff_c^r( V)_\cP$ endowed with the composition law $ f_\cP \circ g_\cP := ( f_p \circ g_p )_{ p \in \cP }  $ are groups. 

 Note that we have: 
\begin{fact}
\label{morphism cP}
The following inclusion is a morphism of topological group:
\begin{equation*} 
  f_\cP = ( f_p) _{p \in \cP} \in \Diff^\infty_0 ( V ) _\cP  \mapsto  \widehat {f_\cP} \in \Diff^\infty_0 ( V \times \cP )\; .\end{equation*}   
\end{fact}
The latter fact will be used among some proofs of the parametric counterparts. 

From now on we will work with $ V = \T \times M$. 

\subsubsection{The space $\P$ of Pluggable maps} 
\label{lieu pluggable}
\emph{Every time, but when explicitly stated, we will focus on the case $r=\infty$}: plugins will be of class $C^\infty$ and their outputs as well. 
Recall that we endow $\T=\R/\Z$ with the  Euclidean Riemannian metric, $M$ and $\cP$  with their Riemannian metric. The product spaces $\T\times M$, $\T\times M\times \cP$ and $M\times \cP$  are endowed with their product Riemannian metric.   
 \begin{definition}\label{def_pluggable}
A map $G\in \Diff^\infty_c(\T\times M)$ is {\em semi-pluggable} if there is a sequence  $(g_k)_{k\ge 1}$ of plugins  $g_k\in \Diff^\infty(\T\times M)$  with step $2^{-k}  $, so that for every large integer $k$ 
the output of $g_k$ is  $G$ and  $g_k\to \id$ for the $C^\infty$-topology when $ k \to \infty$. 
The map $G$ is {\em pluggable} if $G$ and $G^{-1}$ are semi-pluggable. Let: 
\begin{equation*} \P:= \{G\in   \Diff^\infty_c(\T\times M): G \text{ is pluggable}\}\; .\end{equation*}  
\end{definition}
\begin{example}\label{exem id}
The identity $\id$ of $\T\times M$ is pluggable since it is the output of every plugin of the sequence  $ { \left( ( \theta, y ) \in \T \times M \mapsto ( \theta + 2^{ -k} , y)  \right)_{k \ge 1} } $ as in \cref{exem id0}.
\end{example}
We give in \cref{sec:Examples of pluggable dynamics} more sophisticated examples of pluggable dynamics.
Note that in \cref{def_pluggable},  we  are only interested by plugins whose output is compactly supported (outside of the boundary). This is because not every mapping of  $\Diff^\infty_0(\T\times M)$ is pluggable by \cref{cexample1}, while we will show that 
every map in $\Diff^\infty_c(\T\times M)$ is pluggable:
\begin{theo}\label{mainplug}We have $  \P= \Diff^\infty_{c}(\T\times M)$. \end{theo}
Observe that the first assertion of  \cref{main} in the case $r=\infty$ is an immediate consequence of \cref{mainplug}. 
\paragraph{Parametric counterpart.}
The proof of \cref{mainplug}  is (basically) constructive and depends smoothly on the output. For the sake of completeness, we will verify this by giving the parametric counterpart of each statement. Some of the proofs will be designed to be verbatim the same. For a first reading of the proofs, we advise the  reader to skip the all  parametric counterpart of the arguments. 
Here is the parametric counterpart of the notion of plugin:
\begin{definition} A family $g_\cP= (g_p)_{p\in \cP} \in  \Diff^\infty(\T\times M)_\cP$ defines a {\em $\cP$-plugin} if the diffeomorphism: 
$$\widehat { g_\cP} : (z,p)\in M\times \cP\mapsto (g_p(x), p)$$
in $\Diff^\infty( \T \times   M  \times \cP)$ is a plugin. 
\end{definition}
Similarly we have the parametric counterpart of the output of a $ \cP$-plugin: 
\begin{definition}\label{def.outputcP}
The {\em  output} of a $\cP$-plugin $g_\cP \in \Diff^\infty(\T\times M)_\cP$ of step $\sigma$ is the  family $G_\cP $ such that  for each $ p \in \cP$, the map  $G_p$  is the  rescaling of the return map $g_{p}^{\tau_p} \colon\Delta_\sigma \to \Delta_\sigma$:
\begin{equation*} G_p := H_\sigma\circ g_p^{\tau_p}\circ H^{-1}_\sigma\colon \T\times M\to \T\times M.\end{equation*} 
\end{definition}
By \cref{p.diffeocontinuity} the output  $\widehat{ G_\cP} $ of the plugin $\widehat{g_\cP}  $  on $ \T \times M \times \cP $ is smooth and so:
\begin{fact}
The output of a $ \cP $-plugin of $ \T \times M $ lies in $ \Diff_0^r ( \T \times V)_{ \cP } $.
\end{fact}
In the parametric setting, \cref{def_pluggable} becomes:

 \begin{definition}
A family $ G_\cP   \in \Diff_c^\infty( \T \times   M  )_\cP$ 
 is {\em $\cP$-semi-pluggable} if there is a sequence $\left( g_{k\cP} \right)_{k \ge 1} $ of $\cP$-plugins  with step $2^{-k}$, so that for every large  $k$, 
the output of $g_{k\cP} $ is  $G_\cP$ and  $\widehat{ g_{k\cP} }\to \id_{\T \times M \times \cP} $ for the $C^\infty$-topology. 
The family of maps $ G_\cP $ is {\em $\cP$-semi-pluggable} if $ G_\cP $ and $ G_\cP^{-1} := (G^{-1}_p)_{p\in \cP}$ are $\cP$-semi-pluggable. Let: 
\begin{equation*} \P_\cP:= \{G_\cP \in   \Diff^\infty_c(\T\times M)_\cP: G_\cP  \text{ is $\cP$-pluggable}\} \; .\end{equation*} 
\end{definition}
To obtain \cref{mainprime} in the case $r=\infty$, it suffices to prove:
\begin{theoprime}\label{mainplug cP}We have $  \P_\cP= \Diff^\infty_c(\T\times M)_{\cP}$. 
\end{theoprime} 

\subsubsection{Examples of pluggable dynamics} \label{sec:Examples of pluggable dynamics}
Consider the following subgroup of $ \Diff_c^\infty ( \T \times M)$:

\[ \G_1:=  \{ (\theta, y) \in \T \times M    \mapsto  (\theta+ 
 \nu(y) , y ): \nu\in C^\infty_c(M, \R)\} \; . \]
This gives a first example of a subgroup of pluggable maps: 
\[  \G_1 \subset \spp \; . \]
\label{section:twist}
\begin{proposition} \label{twist}The group $ \G_1$ is included in $\P$.  
 \end{proposition}
 
 The subgroup $\G_1$ was first studied in \cite{BT22} into a set of generators of symplectomorphisms. 
 \begin{figure}[H]
  \begin{center}
 {\footnotesize \def\svgwidth{140pt}           %\'echelle du dessin
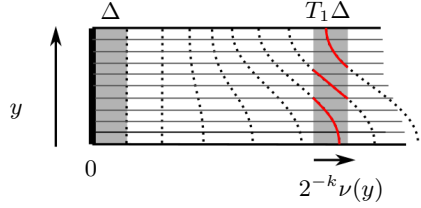}
 \end{center}
\caption{Dynamics of an element of $\G_1$.}
\end{figure}
 
 \begin{proof} 
Let $ \rho\in C^\infty(\T, \R^+)$  be a  function  with support contained in $ [ \tfrac 13,\tfrac 23 ] \subset \T  $ and integral $1$. 
For $\epsilon\in [-2^{-k-2}, 2^{-k-2}]$, we define the smooth vector field:  
\begin{equation}  X_\epsilon: \theta\in \T \mapsto  2^{-k}/(1 - \epsilon \cdot \rho(\theta)) \; . \end{equation}   
Let $\phi^t_\epsilon$ be the flow of $X_\epsilon$. The time taken to go all around the circle equals to:  
\begin{equation} \tau(\epsilon)= \int_\T \frac1{X_\epsilon(\theta)} d\theta=
 \int_\T 2^{ k} \cdot  (1-\epsilon \cdot \rho(\theta)) d\theta= 2^{ k}\cdot  (1-\epsilon)\; .\end{equation}  
  If we stop at time $2^k$, then the lacking or exceeding time for a complete lap is $2^k\epsilon\in [-\tfrac14 , \tfrac14 ]$. As near $0$ the vector field $X_\epsilon$ equal $2^{-k}$, this implies that the image of $0$ by $\phi^{2^k}_\epsilon$ is equal to $\epsilon$. So for every $ \nu\in C^\infty_c(M, \R)$, for every $k$ large,   the map:
\begin{equation} g_k:= (\theta, y)\mapsto (\phi^1_{2^{-k}\nu(y)}(\theta), y)\end{equation} 
coincides with $R_{2^{-k}}$ on  a set which contains the complement of $S=[\tfrac14, \tfrac34]\times   \supp \nu$. Thus  $g_k$  is a plugin with step $2^{-k}$.  Furthermore, by the above discussion, its output equals to:
\begin{equation} G:  (\theta, y)   \mapsto  \left(\theta+ 
 \nu(y) , y  \right) \text{ with } \nu\in C^\infty_c(M, \R)\; .\end{equation}   
Finally observe that $g_k$ is $C^\infty$-close to identity. 
Thus $G$ is semi-pluggable. Hence  $\G_1$ is formed by semi-pluggable maps, and as $\G_1$ is a group, it is formed by pluggable maps. 
\end{proof}

\paragraph{Parametric counterpart.}  Let $\G_{1\cP}$ be the subset of $\Diff_c^\infty(\T\times M)_\cP$ formed by families $(G_p)_{p\in \cP}$ such that $G_p\in \G_1$ for every $p\in \cP$. 

We have similarly:
\begin{proposition} \label{twist cP} The group $\G_{1\cP}$ is included in $\P_\cP$. \end{proposition}

 \begin{proof}
An element $(G_p)_{p\in \cP}\in  \G_{1\cP}$ is formed by mapping of the form 
$G_p: (\theta, y)   \mapsto  (\theta+  \nu_p(y) , y )$ with $(\nu_p)_{p\in \cP}$ smooth. Thus   
the vector field $Z_p:(\theta,y)\mapsto (X_{2^{-k} \nu_p}(y), 0)$ depends smoothly on $p$.   Hence the time one  maps $g_p$ form a family of plugins $(g_p)_{p\in \cP}$  (with output $G_p$ and step $2^{-k}$) which depends smoothly on $p\in \cP$.   \end{proof}

\section{Topological group structure on $\P$}  \label{sec top group}
\subsection{Group structure on $\P$}\label{group section}
In this section we show that $\P$ endowed with the composition rule is a group. To this end,   let  $ \pi: \bar \T\times M:= \R/2\Z \times M \to \T\times M= \R/\Z \times M$  be the canonical 2-sheeted covering map  and denote by 
$\psi:  (\theta,y) \in  \bar\T\times M\mapsto  (\theta/2, y)\in\T\times M$
   the canonical diffeomorphism. 
    The following   defines a binary operation $\star$ on the space of plugins:
  \begin{definition}
 Let   $g_1 $ and $g_0 $ be two plugins with same step $\sigma=2^{-k}$.  Let ${\bar g_1, \bar g_0\in \Diff^\infty (\widehat\T\times M)}$ be the lifts of $g_1$ and $g_0$ such that 
 $\bar{g}_1(0,y)=\bar {g}_0(0,y)=(\sigma,y)$ for every $y\in  M$.
 Let $\overline{g_1\star g_0}$ be equal to  the lift  $\bar g_0$ on  the first half of $\bar\T\times M$  and be equal to the  lift of $\bar g_1$ on second half of $\bar\T\times M$:
\begin{equation*} \overline{g_1\star g_0}
: (\theta, y)\in \bar \T \times M \mapsto \left \{\begin{array}{cl} \bar g_0(\theta,y)& \text{if } 
 \theta\in  [0,1)+2\Z\; , \\ 
\bar g_1(\theta,y),& \text{if } 
 \theta\in [1,2)+2\Z\; ,  \end{array}\right. \text{ and put } g_1\star g_0:=   \psi \circ \overline{g_1\star g_0}\circ \psi^{-1}   \; .\end{equation*} 

\end{definition}
  
\begin{remark}\label{p.continuityatId}
Given a neighborhood $V$ of $\id$ in $\Diff^r_0(\T\times M)$, there exists a neighborhood $W$ of $\id$ such that for any pair of plugins $f, g\in W$ of same steps, we have $f\star g\in V$.
\end{remark}

\begin{figure}[H]\begin{center}
 {\footnotesize 
\begin{equation*} \raisebox{-2ex}{\def\svgwidth{30pt}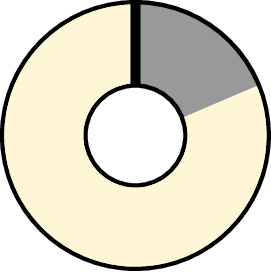}%\pause
  \; \xrightarrow{\begin{array}{c}\mbox{$2$-sheeted}\\\mbox{lift}\end{array}}\; \raisebox{-5ex}{\def\svgwidth{50pt}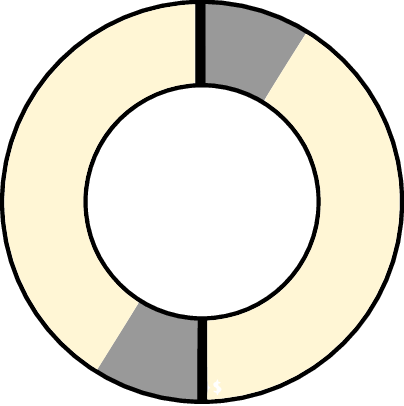}%\pause 
  \quad\xrightarrow{\begin{array}{c}\mbox{cut \&}\\\mbox{paste}\end{array}}\quad \raisebox{-5ex}{\def\svgwidth{50pt}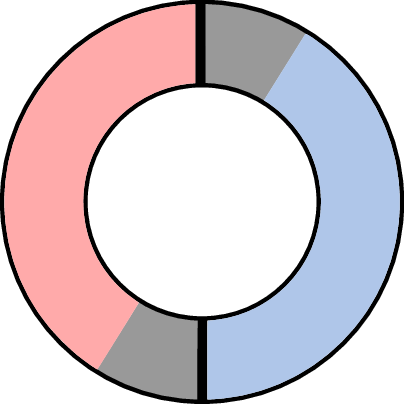}%\pause
  \quad\xrightarrow{\theta/2}\quad\raisebox{-2ex}{\def\svgwidth{30pt}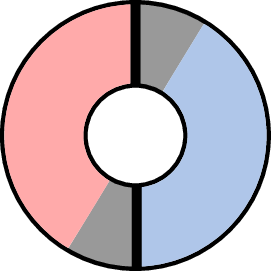}\end{equation*} }
  \end{center}
  \caption{Concatenation of two plugins.}\label{conca plug}
\end{figure}

The $\star$-product associates to a pair of plugins of the same step  a plugin of half that step and whose  output is the composition of the  outputs: 
\begin{proposition}\label{outputstar}
If $g_0$ and $g_1$  are plugins with a same step $\sigma$ and outputs $G_0$ and $G_1$ then 
$g_1\star g_0$ is a plugin with step $\sigma/2$ and  output  $G_1\circ G_0$.   
\end{proposition}
\begin{proof} %\emph{$g_2\star g_1$ is a plugin.}  
The set $\Delta=([0,\sigma)+\Z)\times M$ lifts into the union of the two sets:
\begin{equation*} \bar{\Delta}_0=([0,\sigma)+2\Z)\times M \qand
\bar{\Delta}_1=([1,1+\sigma)+2\Z)\times M.
\end{equation*} 
Let $\pi_i$ be the restriction of $\pi$ to each $\bar{\Delta}_i$.

Note that  $\bar g_0$ and $\bar g_1$ coincide with the  translation by $(\sigma  ,0)$ on $\bar{\Delta}_0\cup \bar{\Delta}_1$. Hence the map $\overline{g_0\star g_1}$ is a smooth diffeomorphism that also coincides  with the  translation by $(\sigma  ,0)$ on  $\bar{\Delta}_0\cup \bar{\Delta}_1$.
Then   iterations of $\overline{g_1\star g_0}$    send  $\bar{\Delta}_0$ onto $\bar{\Delta}_1$ by $\bar g_0^{\tau_0\circ \pi}$ and send  $\bar{\Delta}_1$ onto $\bar{\Delta}_0$ by $\bar g_1^{\tau_1\circ \pi}$. Thus we have:
\begin{equation*}  {\bar g_{ 0 } ^{\tau_0\circ \pi}}{\vert \bar \Delta_0} = \pi_1^{-1}\circ H_{\sigma}^{-1}\circ G_0\circ H_{\sigma}\circ \pi_0 \qand 
\bar g_{1 }^{\tau_1\circ \pi} \vert \bar \Delta_1= \pi_0^{-1}\circ H_{\sigma}^{-1}\circ G_1\circ H_{\sigma}\circ \pi_1 \; . \end{equation*}  

Therefore the return time of $\overline{g_1\star g_0}$ into $\bar{\Delta}_0$ is defined on $ \bar \Delta_0 $ and the first return map is: 
\begin{equation*} \pi_0^{-1}\circ H_{\sigma}^{-1}\circ G_1\circ G_0\circ H_{\sigma}\circ \pi_0.\end{equation*} 
This implies that $g_1\star g_0$ is a plugin with output $G_1\circ G_0$ and step $\sigma/2$.
\end{proof}

  \begin{proposition}\label{group}
  The set $\P$ is a  subgroup of $ \Diff^\infty_c(\T\times M)$ endowed with the composition rule.
  \end{proposition}
\begin{proof}  
By  \cref{exem id}, $\P$ contains the identity and by definition it is stable by inversion. Thus it remains to show that semi-pluggability is stable by composition.  
Let $G_1, G_0$ be semi-pluggable. Then for every  $k$ large,   $C^\infty$-close to identity, there are plugins $g_{1 }$ and $g_{0 }$  with step $ 2^{-k}$ and with outputs $G_1$ and $G_0$ respectively. Then by \cref{p.continuityatId} , the map    $g_{1 }\star g_0$ is close to identity when $k$ is large. By \cref{outputstar}, the output of $g_{1 }\star g_0$ is $G_1\circ G_0$ and its step is $2^{-k-1}$. Thus $G_1\circ G_0$ is semi-pluggable. The second assertion is proved similarly.
 \end{proof}  
 \paragraph{Parametric counterpart.}
The proofs of the two latter propositions imply immediately:
\begin{proposition}
If $f_\cP ,  g_\cP$ are two $ \cP$-plugins with same step $\sigma$ and output $ F_\cP $ and $  G_\cP $, then $(f_p\star g_p)_{p\in \cP}$ is a $\cP$-plugin with step $ \sigma/2$ and output $ (F_p \circ  G_p) _{ p \in \cP} $. 
\end{proposition}
Thus we deduce: 
\begin{proposition}\label{group cP}The set   $\P_\cP$  is a  subgroup of $ \Diff^\infty_c(\T\times M)_\cP$ endowed with operation:
\begin{equation*} (F_p)_{p\in \cP}\circ (G_p)_{p\in \cP}=
(F_p \circ  G_p)_{p\in \cP}\; .\end{equation*}   
  \end{proposition}

 \subsection{Another subgroup included in $\P$}
 \label{section:Dil}
Consider the  following sub-group  of $\Diff^\infty_{ c}(\T\times M)$:
\begin{equation*} \G_2:= \left\{ (\theta, y)\in \T\times M   \mapsto (\theta,F(y)): 
F \in \Diff_c^\infty(  M)\right\}\end{equation*} 
\begin{proposition}\label{Dil}  The  group $\G_2$ is included in $\P$.  
\end{proposition}

 \begin{figure}[H]
  \begin{center}
 {\footnotesize \def\svgwidth{110pt}           %\'echelle du dessin
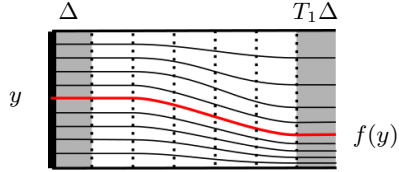}
 \end{center}
 \vskip -5mm
\caption{Dynamics of a plugin with output in $\G_2$.}
\end{figure}

\begin{proof}
 As $\G_2$ is connected  and $\P$ is a group by \cref{group}, it suffices  to show that a neighborhood $W$ of $\id$ in  $\G_2$  is included in  $\P$.  Indeed
by \cite[Prop 3.18]{willmore1984foundations}, any such neighborhood $W$ generates $\G_2$.
Up to replacing $W$ by $W\cap W^{-1}$, it suffices to show that any element of $W$ is semi-pluggable.

In order to do so,  we develop an  idea which appears in \cite{NRT78}.  Take $W$ sufficiently small so that for every $G:(\theta,y)\mapsto (\theta, F(y))\in W$, the map $F$ is sufficiently close to $\id$ to be isotopic to it via a smooth path. In other words, there exists a $C^\infty$-family  $(F_t)_{t\in [0,1]}$ of maps $F_t\in \Diff^\infty_c(M)$ such that $F_0=\id$ and $F_1=F$.
Such a family can be obtained using the exponential map $\exp$ of the Riemannian metric,  via the formula $F_t:= y\mapsto  \exp_y(t \cdot \exp_y^{-1} F(y))$. Define:
\begin{equation} Y(t,y):=  \partial_t F_t\circ F_t^{-1}(y)\; , \end{equation} 
and observe that $F$   is the time one map of the (compactly supported)  non-autonomous   vector field $Y $.
 Let $\tau: \T\to [0,1]$ be a map which is smooth on $\T\setminus \{0\}$ and such that near $0^+$ it equals $0$ and near $0^-$ is equal $1$.
 For  $k$ large, let:
\begin{equation} X_k: (\theta ,y) \in\T\times M \mapsto ( 2^{-k}, 2^{-k} \cdot\partial_\theta \tau(\theta)\cdot  Y(\tau(\theta), y))\; .\end{equation} 

Let $g_k$ be the time one map of this vector field.
Observe that for $k$ large enough,  $g_k$ is a plugin with step $\sigma=2^{-k}$ and return time $\sigma^{-1}$. Furthermore, its output is $G$. Indeed the second coordinate of the output is the time  $ \sigma^{-1} $ map  of the flow of $\partial_\theta \tau(\theta)\cdot  Y(\tau(\theta), y)$ which is the time one map $F$ of the flow of $Y$. Thus we have:
\begin{equation}\label{preplugin} g^{1/\sigma}_k\circ H_\sigma (x)= H_\sigma \circ   G(x)   , \quad \forall x \in \Delta_\sigma \; .\end{equation}
Furthermore, when $k$ is large, the plugin $g_k$ is close to identity. Thus $G$ is semi-pluggable. 
\end{proof}

\paragraph{Parametric counterpart.} Let  $\G_{2\cP}$ be the subset of $\Diff_c^\infty(\T\times M)_\cP$ formed by families $(G_p)_{p\in \cP}$ such that $G_p\in \G_2$ for every $p\in \cP$. The following is a counterpart of \cref{Dil}:
 \begin{proposition}\label{Dil cP}
 The group $ \G_{2\cP}$  is included in $\P_\cP$. 
 \end{proposition} 
 
 \begin{proof} 
Let $W$ be as defined in the proof of \cref{Dil}. We define $W_\cP$ as  the subset of $\Diff_c^\infty(\T\times M)_\cP$ formed by families $(G_p)_{p\in \cP}$ such that $G_p\in W$ for every $p\in \cP$. For the same reasons it suffices to show that any element of $W_\cP$ is $ \cP$-semi-pluggable. 
Similarly, for any family $( G_p)_{p \in \cP} =: (id_\T\times F_p)_{p\in \cP} \in W_{\cP} $
, the family $( F_p)_{ p \in \cP} $ is isotopic to the identity via a smooth path $ \left(  ( F_{pt})_{ p \in \cP}  \right)_{ t \in [ 0,1]}$  where $F_{pt} := y\mapsto  \exp_y(t \cdot \exp_y^{-1} F_p(y))$. We define:
\begin{equation} Y_p (t,y):=  \partial_t F_{pt}\circ F_{pt}^{-1}(y)\; .  \end{equation} 
Note that the family of vector fields $(Y_p)_{p \in \cP} $ is smooth. Define then the family of vector fields:
 \begin{equation} X_{kp} : (\theta ,y) \in\T\times M \mapsto ( 2^{-k}, 2^{-k} \cdot\partial_\theta \tau(\theta)\cdot  Y_p(\tau(\theta), y))  \; , \end{equation} 
 where $ \tau $ is the function  defined in \cref{Dil}.
For $k$ large enough, the family of time one maps $(g_{kp})_{p \in \cP} $ is a $\cP$-plugin with output the family $(G_p)_{p \in \cP}$. Moreover, when $k\to \infty$, the $\cP$-plugin $(g_{kp})_{p \in \cP}$ tends to the identity $ \id \in \Diff^\infty_c(M)_\cP$.
\end{proof}

\subsection{Closedness of the group $\P$} \label{subsec close}
In this section we prove that $\P$ is closed in $ \Diff^\infty_c(\T\times M)$.  

\begin{proposition}\label{c.closure}
The subgroup  $\P\subset \Diff^\infty_c(\T\times M)$ is closed. 
\end{proposition}

This proposition uses the following lemma proved below:
\begin{lemma}\label{l.neigh.of.Id} 
For any $ 1 \le r \le \infty  $,
for any neighborhood $\cN$ of  $\id\in \Diff^r (\T\times M)$, there exist $N\ge 1$ and a neighborhood $\cN_c$ of $\id\in \Diff_c^r (\T\times M)$ such that for 
 all  $G\in \cN_c$ and   $k\ge N$, there is a $C^r$-plugin $g\in \cN$ with output $G$ and step~$2^{-k}$. 
\end{lemma} 

Note that the latter lemma is redacted for any regularity $1 \le r \le \infty$. It will allow to deduce \cref{main} from \cref{mainplug}.  

\begin{proof}[Proof of \cref{c.closure}] 
It suffices to show that the set of  semi-pluggable maps is closed. Indeed, the continuity of the involution $G\mapsto G^{-1}$ implies that the set of maps with   semi-pluggable  inverse is closed; and so it comes that the intersection $\P$ of these two sets is closed. 

Let $(G_j)_{j \geq 0 }$ be a sequence semi-pluggable  maps  converging in the $C^\infty$-topology to a diffeomorphism $ G \in \Diff_c^\infty(\T\times M) $.
Let us show that $G$ is semi-pluggable. 
 In other words, let us show that for every neighborhood $\cV$ of  $\id\in \Diff^\infty (\T\times M)$, for every $k$ large enough, there exists a plugin $g\in  \cV$ with output $G$ and step $2^{-k}$. 

To this end, let us fix a small neighborhood $\cN$ of  $\id \in \Diff^\infty (\T\times M)$ and $j$ large so that $  G_j^{-1}\circ G$ belongs to the open set $\cN_c$ given by \cref{l.neigh.of.Id}. Hence for every $k\ge 0$ large enough, there exists a plugin $g_1\in \cN$ with output  $ G_j^{-1}\circ G$ and step $2^{-k}$.
As $G_j$ is pluggable, for every $k$ large enough,  there exists a plugin $g_0\in \cN$ with output  $G_j$ and step $2^{-k}$.
 Now we merge the plugins $g_1$ and $g_0$ to obtain a plugin $g= g_1\star g_0$ of $G$ of step $2^{-k-1}$. By \cref{p.continuityatId}, when $\cN$ is small, the map $g$ is close to identity and so in  $\cV$. 
\end{proof}

 The idea of the proof of  \cref{l.neigh.of.Id}  is to find, for each fixed small $\delta>0$,  a sequence of close to identity plugins  $g_k$ with step $2^{-k}$ and output the identity such that some iterates of each $g_k$ stretch the $2^{-k}$-thin fundamental domain $\Delta_{2^{-k}}$ onto a wider fundamental domain, isometric to $[0,2\delta[\times M$. The iteration by $g_k$ will produce a horizontal zooming effect on $\Delta_{2^{-k}}$.  Then we will be able to perturb the plugin on this stretched fundamental domain, to obtain an open set of outputs independent of~$k$.

 The following produces the sequence $(g_k)_k$:
  \begin{sublemma} \label{s.zoomin}
For every neighborhood   $\cN$   of $\id\in \Diff^\infty(\T \times M)$, for every  $\delta>0$ small,  there exists $N\ge 1$ and  a sequence $(g_k)_{k\geq N}$ of plugins in $\cN$  with step $2^{-k}$, output $\id$ such that:  
\begin{itemize}
\item for all $y\in M$ and $\theta\in  [\tfrac12,\tfrac23]$;   we have  $g_k(\theta,y)=(\theta+ \delta ,y),$ 
\item $g_k$ is of the form $g_k: (\theta, y)\in \T\times M \mapsto  (\phi_k(\theta), y)$ where $\phi_k$ is the time-$1$ map of a flow.  
\end{itemize}
\end{sublemma}
\begin{proof} Let $A= [\tfrac12, \tfrac23]$ and $B= [\tfrac 14, \tfrac13]$. Let $\psi_A$, $\psi_B\in C^\infty(\T, [0,1]) $ be two non-negative functions with disjoint supports, vanishing at a neighborhood of $0$ and  such that:
\begin{equation}  \psi_A |A  =1 \qand \psi_B |B  =1 \; . \end{equation} 
For $ \beta\ge 0$ we define the following vector field on the circle $\T$:
\begin{equation} X_{ \beta,k}:= \delta\cdot \psi_A+\beta\cdot \psi_B+(1-\psi_A-\psi_B)\cdot 2^{-k}.\end{equation} 
Let $\tau_{  \beta, k}$ be the time needed to make one turn around the circle along the flow of $X_{ \beta}$.  This number is large since $\delta$ is small. Note that $\tau_{ \beta, k}$ depends smoothly on $  \beta>0$. Also $\tau_{ \beta, k}\to \infty$ when $\beta\to 0$ and  $\partial_\beta \tau_{ \beta, k}<0$. 
Let $N\ge 1$ be so that $\delta>2^{-N}$. Let $k\ge N$. We have  $\delta>2^{-k}$. If  $\beta=2^{-k}$, then the time $\tau_{ \beta, k}$ is smaller than $2^k$. Thus by the mean value theorem, there exists a unique $\beta= \beta(k,\alpha)$ close to $2^{-k}$ such that $\tau_{\delta, \beta(k,\delta), k}=2^k$. 

Then the time $1$ map $g_k$ of the flow of $(X_{  \beta(k),k},0)$   satisfies the desired properties.\end{proof}
We have now the tools to prove the following restricted version of  \cref{l.neigh.of.Id}:

  \begin{sublemma}\label{claim.neigh.of.Id} 
  For any $ 1 \le r \le \infty$, for any neighborhood $\cV$ of  $\id\in \Diff^r (\T\times M)$, there exist $N\ge 1$ and a neighborhood $\cN_c$ of $\id\in \Diff_c^r (\T\times M)$ such that for any $k\ge N$ and
 every  $G\in \cN_c$ whose restriction to a neighborhood of $\{0\}\times M$ or a neighborhood of $\{\tfrac12\}\times M$ is the identity, there is a $C^r$-plugin $g\in  \cV$ with output $G$ and step~$2^{-k}$.  
\end{sublemma}

\begin{proof}[Proof of \cref{claim.neigh.of.Id} ]
Let $\cN$ be a neighborhood of $\id\in \Diff^r(\T\times M)$ such that\footnote{in the sense that the distance between $\cN$ and the complement of $\cV$ is positive.} $\cN\Subset \cV$. We apply \cref{s.zoomin} which provides $\delta>0$,  $N\ge 1$ and a sequence $(g_k)_{k\ge N}$  of plugins $g_k:(\theta, y)\mapsto  (\phi_k(\theta), y)$ with output $\id$ and step $2^{-k}$. 
By \cref{s.zoomin}, there exists  $n_k>0$ minimal such that $\theta_k:=\phi^{n_k}_k (0) \in [\tfrac12, 1]$.  Taking $\cN$ small, we have that  $\theta_k$ is smaller than $\tfrac23 - \tfrac32 \delta$, and so $g_k$ equals the translation by $\delta$ on  $(\theta_k, \theta_k+\tfrac32 \delta)$. 
Let $\cN_c$ be a small neighborhood of  $\id\in \Diff_c^r(\T\times M)$. Let $G\in \cN_c$ be equal to the identity  near  $\{0\}\times M$ or  $\{\tfrac12 \}\times M$. We would like to $C^r$-perturb $g_k$ so that its output is $G$.  

\emph{Case 1:} If  $G\in \cN_c$ coincides with the identity near $\{0\}\times M$, then we  perform a perturbation of  $g_{k}$ supported by  $(\theta_k ,\theta_k+\delta)\times M$ and therein equals to:
\begin{equation}  \begin{array}{rcl}\tilde g_k : [\theta_k ,\theta_k+\delta]\times M&\to& [\theta_k+\delta ,\theta_k+2\delta]\times M\\
   (\theta_k+x,y)&\mapsto &(\theta_k+\delta,0)+\delta\cdot \tilde G(\delta^{-1} x  ,y)\end{array}
       \end{equation} 
where $\tilde G:\R\times M\to \R\times M $  is a lifting of $G$ which fixes $\{0\}\times M$. 
Note that $\tilde g_k$ is a $C^r$-plugin with output $G$ and step $2^{-k}$. Furthermore, if $\cN_c$ is small enough (at $\delta$ fixed), then for every $k\ge N$, the map  $\tilde g_k$ is uniformly close to $g_k\in \cN$ in the $C^r$-topology and so  $\tilde g_k$ belongs to $\cV$.

\emph{Case 2:} If  $G\in \cN_c$ coincides with the identity near $\{\tfrac12\}\times M$, then we  perform a perturbation of  $g_{ k}$ supported by  $(\theta_k +\tfrac\delta2  ,\theta_k+\tfrac32\delta )\times M$ and therein equals to:
   \begin{equation} \begin{array}{rcl}\tilde g_k :[\theta_k +\tfrac\delta2,\theta_k+\tfrac32\delta ]\times M &\to& [\theta_k +\tfrac32\delta,\theta_k+\tfrac52\delta ]\times M\\
   (\theta_k +\tfrac\delta2+x,y)&\mapsto& (\theta_k +\tfrac\delta2+\delta,0)+\delta\cdot \tilde G(\delta^{-1} x  ,y) \; . 
   \end{array}
      \end{equation} 
Similarly, this is a $C^r$-plugin of output $G$ and step $2^{-k}$, which is in $\cV$ for every $k$ provided that $\cN_c$ is small enough.  
\end{proof}  

\begin{proof}[Proof  that \cref{claim.neigh.of.Id}  implies \cref{l.neigh.of.Id} ] 
First let us `fragment' any $C^r$-close to identity map $G\in \cN_c$ into the composition of two $C^r$-maps $G_1\circ G_0$ such that $G_0$ coincides with the identity near $\{0\}\times M$  and 
 $G_1$ coincides with the identity near $\{\tfrac12\}\times M$.

To this end, we use the exponential map $\exp$ associated to the geodesic flow of $\T\times M$ and a function $\rho\in C^\infty(\T\times M, [0,1]) $ such that $\rho|\{0\}\times M=1$ and $\rho|\{\tfrac12 \}\times M =0$.  Let   $\cN_c$ be a sufficiently small neighborhood of $\id\in \Diff^\infty_c(\T\times M)$   such  that for every $G\in \cN_c$ the following is a smooth diffeomorphism:
\begin{equation} G_0:= x\mapsto \exp(\rho (x) \cdot \exp_x^{-1}G(x)) \; . \end{equation} 
Let $G_1:= G\circ G_0^{-1}$ and note that  $G= G_1\circ G_0$. 

\cref{claim.neigh.of.Id} states that there are $C^r$-plugins $g_1$ and $g_0$  close to identity with step $2^{-k}$ and outputs $G_1$ and $G_0$ for every $k$ large enough. Then by \cref{p.continuityatId} and \cref{outputstar}, the $C^r$-plugin $g_1\star g_0$ of step $2^{-k-1}$ is close to identity  and has output $G_1\circ G_0=G$.
\end{proof}

 \begin{proof}[Proof that \cref{mainplug} implies \cref{main} \label{lieu proof main}]
 When $r = \infty$, the result of \cref{main} corresponds to the one of \cref{mainplug}. Consider now $ r < \infty$. Let $ G \in \Diff_c^r( \T \times M) $, and  $ \cN \subset \cV $ two neighborhoods of $ \id \in \Diff_c^r( \T \times M) $. We smooth the map $ G $ into a map $ \tilde G \in \Diff_c^\infty( \T \times M)  $ such that the map $ \tilde G^{-1} \circ G $ belongs to the neighborhood $ \cN_c$ given by \cref{l.neigh.of.Id}. Then for every $ k \ge 0$ large enough, there exists a plugin $ g_0 \in \cN$ with output $ \tilde G^{-1} \circ G $ and step $ 2^{-k} $. By \cref{mainplug}, for $k$ large enough, there exists a plugin $ g_1  \in \cN$ with output $ \tilde G$. We merge $g_0$ and $ g_1$ to get a plugin $ g = g_1 \star g_0 $ with output $G$. By \cref{p.continuityatId}, when $ \cN $ is small, the map $g$ is in $ \cV$. 
 % When $r = \infty$, the result of \cref{main} corresponds to the one \cref{mainplug}. For the case when $r$ is finite, we first smooth the map $ G$ to obtain a map $ \tilde G  \in \Diff_c^\infty( \T \times M )$ arbitrarily close to $G$ for the $C^r$-topology. By \cref{mainplug} there exists a sequence of plugins $ (\tilde g_k)_{k \ge 1}  $ with step $2^{-k} $ such that the output of $ \tilde g_k $ is $\tilde G$ for $k $ large enough and $  \tilde g_k $ tends to $ \id$ in the $C^\infty$-topology, when $ k$ approaches infinity. Observe now that $ F :=  G \circ \tilde G^{-1}   $ is close to identity. So we can apply \cref{l.neigh.of.Id} which provides a sequence of $C^r$-plugins $( f_k)_{k \ge 1} $with output $ F$ for $k$ large enough and such that $f_k \xrightarrow[k \to \infty]{C^r} \id $. 
 % For $k \ge 1 $ set $ g_k := f_k \star \tilde g_k$. Finally observe that the sequence $(g_k)_{k \ge 1} $ satisfies the conditions of  \cref{main}. 
 \end{proof}

 \paragraph{Parametric  counterpart.} \label{lieu cas CR}
Here is the parametric  counterpart of \cref{c.closure}: 
 \begin{proposition}\label{c.closure cP}
The set $\P_\cP$ is closed. 
 \end{proposition}
To show this proposition we will use the following counterpart of \cref{l.neigh.of.Id} proved below: 
 \begin{lemma}
 \label{l.neigh.of.Id cP}
For any $ 1 \le r \le \infty  $,
for any neighborhood $\cN $ of  $\id\in \Diff^r (\T\times M)$, there exists $N\ge 1$ and a neighborhood $\cN_{c} $ of $\id\in \Diff_c^r (\T\times M)_\cP$ such that for 
 all  $ ( G _p )_{ p \in \cP}  \in \cN_{c } $ and   $k\ge N$, there is a $ \cP$-$C^r$-plugin $(g_p)_{p \in \cP} \in \cN $ with output $( G _p )_{ p \in \cP}$ and step~$2^{-k}$. 
 \end{lemma}

\begin{proof}[Proof of \cref{c.closure cP}]
We proceed literally as in the proof \cref{c.closure}, by applying \cref{l.neigh.of.Id cP}  instead of \cref{l.neigh.of.Id}, and considering families instead of single maps. Note that the continuity of $\star $ given by \cref{p.continuityatId} is also valid for families since the embedding ${(f_p)_p\in  \Diff^\infty (\T\times M)_\cP\mapsto f_\cP\in  \Diff^r (\T\times M\times \cP)}$ commutes with the $\star$-product. 
\end{proof}

\begin{proof}[Proof of \cref{l.neigh.of.Id cP}]
The explicit construction of \cref{claim.neigh.of.Id} gives directly a parametric counterpart of this lemma. Indeed note that the maps $ \tilde g_k $ defined in its proof depends smoothly on $ G $.
Moreover the maps $ G_0$ and $G_1$ obtained in the proof of \cref{l.neigh.of.Id} by the fragmentation formula  depends smoothly on the involved diffeomorphism.
\end{proof}
Similarly, we show: 
  \begin{proof}[Proof that \cref{mainplug cP} implies \cref{mainprime}]
  This goes exactly the same as for the proof that \cref{mainplug} implies \cref{mainprime}. It suffices to replace maps in $ \Diff^r_c ( \T \times M)$ by families of maps in $ \Diff^r_c ( \T \times M)_\cP$, plugins by $ \cP$-plugins and \cref{l.neigh.of.Id} by \cref{l.neigh.of.Id cP}.  
% For any family $(G_p)_{p\in \cP}\in \Diff^r_{c}(   \T\times M)_\cP$, any neighborhood  $\cN_c$ of the identity in $ \Diff^r(   \T\times M)_\cP$, it suffices to show the existence of a $C^r$-family of plugins 
% $(g_p)_{p\in \cP} \in \Diff^r(\T\times M)_\cP$ in $\cN_c$ so that the output of $g_p$ equals $G_p$ for every~$p$.  

% When $r=\infty$ this is an immediate   consequence of \cref{mainplug cP}. For $r<\infty$, first we smooth the family $(G_p)_{p\in \cP}$ to a family $(\tilde G_p)_{p\in \cP}\in \Diff^\infty_{c}(   \T\times M)_\cP$ close to $(G_p)_{p\in \cP}$. Then we apply  \cref{mainplug cP} to obtain a family $(\tilde g_p)_{p\in \cP} \in \Diff^\infty(\T\times M)_\cP$ in $\cN_c$ so that the  output of $\tilde g_p$ equals $\tilde G_p$ for every $p$. Observe that  
% $(G_p \circ \tilde G_p^{-1})_{p\in\cP}$ is close to identity. 
% So we can apply \cref{l.neigh.of.Id cP} which provides a $C^r$-$\cP$-plugin $(f_p)_{p\in\cP}$ arbitrarily $C^r$-close to identity with output $(G_p \circ \tilde G_p^{-1})_{p\in\cP}$.
% This family can be chosen sufficiently $C^r$-close to identity so that 
% $(f_p\star \tilde g_p)_{p\in\cP}$ is in $\cN_c$. The output of the latter $\cP$-plugin is  $(G_p)_{p\in \cP}$.  
 \end{proof}

 \subsection{Vector fields whose flow is pluggable}
\label{subsec debut vector} 
 In this section $V$ denotes a manifold. We recall that $\Diff_c^\infty(V)$ endowed with the composition rule $\circ$  is a  Lie group. 
The space $\diff_c^\infty(V)$ of $C^\infty$-vector fields $X$ on $V$ whose support is a compact subset of $V\setminus \partial V$ is a Lie algebra endowed with the Lie bracket:
 \begin{equation*}  [X,Y]: = DY(X)-DX(Y)\;, \quad \forall X,Y\in\diff^\infty_c (V). \; \end{equation*} 
We will work with the Lie algebra counterpart of the considered subgroups. We recall that a subset of $\diff_c^\infty(V)$ is a {\em Lie subalgebra} if it is a vector space  stable by Lie Bracket. To define the counterpart, we will use  the flow $(\mathrm{Fl}_X^t)_t$ of vector fields  $X\in \diff^\infty_c (V)$.

\begin{definition}
We denote $\spp$ the set of  vector fields whose flow is pluggable. In short:
\begin{equation*}  \spp := \lbrace X \in \diff_c^\infty ( \T \times M) : \Fl_X^t \in \P, \ \forall t \in \R  \rbrace \; .  \end{equation*} 
\end{definition}

Using that $\P$ is a closed subgroup, the following is an immediate consequence of  \cref{prop lie 1} of \cref{a.lie}:
\begin{proposition} \label{p.sppclose}
The space $ \spp $ is a closed Lie subalgebra of $ \diff_c^\infty( \T \times M )$.  
\end{proposition}
The following are closed Lie algebras of $\diff_c^\infty ( \T \times M)$: 
\begin{equation*}  \sg_1= \sg_1 ( M)  := \lbrace X \in \diff_c^\infty ( \T \times M) : \Fl_X^t \in \G_1, \ \forall t \in \R  \rbrace = \{ X: (\theta,y)\mapsto (v(y),0): v\in C^\infty(  M , \R)\}
\;  .  \end{equation*} 
\begin{equation*}  \sg_2= \sg_2 ( M)  := \lbrace X \in \diff_c^\infty ( \T \times M) : \Fl_X^t \in \G_2, \ \forall t \in \R  \rbrace = \{ X: (\theta,y)\mapsto (0, f(y)): f\in \diff_c^\infty ( M) \}
\;  .  \end{equation*} 

The  subgroups $\G_1$ and $\G_2$  are in $\P$, so:
\begin{equation} \label{sub12}
    \sg_1 \subset \spp \qand \sg_2 \subset \spp \; .  \end{equation}
\paragraph{Parametric counterpart.}\label{paracaseVect} 
We denote  $\diff^\infty_c(V)_\cP$  the subspace of  families $X_\cP=(X_p)_{p\in \cP}$ in $\diff^\infty_c(V)$ such that:
\begin{equation} \label{def hat X}
\widehat {X_\cP}  \colon (x,p)\mapsto (X_p(x),0)\end{equation} 
is smooth and compactly supported, that is, such that $\widehat {X_\cP}  \in \diff^\infty_{c}(V\times \cP)$. By \cref{morphism cP}, the space $\diff_c^\infty ( V )_\cP $ is a Lie algebra  endowed with the Lie bracket $[X_\cP,Y_\cP]:=([X_p,Y_p])_{p\in \cP}$.
\begin{definition}\label{d.diff.tang.space} The $ \cP$-families of vector fields whose flow is $ \cP$-pluggable is denoted: 
\begin{equation*}   \spp_\cP := \left\{  ( X _ p) _{ p \in \cP } \in \diff_c^\infty ( \T \times M) _{ \cP}  : ( \Fl_{X_p}^t )_{ p \in \cP} \in \P_\cP , \ \forall t \in \R \right\} \; . \end{equation*}  
\end{definition}
Using that $\P_\cP$ is a closed subgroup, the following is an immediate consequence of  \cref{coro lie 1} of \cref{a.lie}:
\begin{proposition} \label{p.sppclose.para}
The space $ \spp_\cP $ is a closed Lie subalgebra of $ \diff_c^\infty( \T \times M )_\cP$.  
\end{proposition}
Also note that the space $ \spp_\cP $ contains:
\begin{equation*}  \sg_{1\cP}  := \lbrace ( X_p)_{ p \in \cP} \in \diff_c^\infty ( \T \times M) _\cP :  X_p \in \sg_1 , \  \forall p \in \cP  \rbrace \end{equation*} 
and 
\begin{equation*}  \sg_{2\cP}  := \lbrace ( X_p)_{ p \in \cP} \in \diff_c^\infty ( \T \times M) _\cP :  X_p \in \sg_2 , \ \forall p \in \cP  \rbrace 
\;  .  \end{equation*} 
We define also:
\begin{equation*}  \widehat {\sg_{1\cP}} :=   \lbrace \widehat{X_\cP} : X_\cP  \in \sg_{1\cP} \rbrace \qand \widehat {\sg_{2\cP}} :=   \lbrace \widehat{X_\cP} : X_\cP  \in \sg_{2\cP} \rbrace \; .    \end{equation*}   
Observe that: 
\begin{equation*}   \widehat {\sg_{1\cP}}  = \lbrace ( \theta , y , p )  \mapsto ( \nu(y,p) , 0 , 0 )   : \nu \in  C^\infty_c ( M \times \cP, \R ) \rbrace  \;   \end{equation*} 
and
\begin{equation*}   \widehat {\sg_{2\cP}}  = \lbrace ( \theta , y , p )   \mapsto ( 0 , f( y,p)  , 0 )   : (f, 0)   \in \diff^\infty_c ( M \times \cP ) \rbrace  \; .   \end{equation*} 
Thus we have $ \widehat{\sg_{1\cP}} = \sg_1 ( M \times \cP )$ and $ \widehat{\sg_{2\cP}} \subsetneq \sg_2 ( M \times \cP )  $. 
\section{Construction of Pluggable flows} \label{sec coeur de la preuve constr of plugg vect}
In this section we show \cref{mainplug} by proving that any vector field has a pluggable flow. In order to do so we will show that the following subspace of $\diff^\infty_c( \T \times M) $ is  in $\spp$:
% \begin{equation*}  \sg_3 = \sg_3 ( M) :=\left\{ (\theta,y) \mapsto \sum_{1\le i \le N}  
% \phi_i ( \theta ) \cdot g_i (\theta , y): N\ge 1\;,\;  \phi_i\in  C^\infty ( \T , \R ) \qand   g_i \in \sg_2\right\}\; .\end{equation*} 
\[ \sg_3 := \sg_3 ( M) = \lbrace X \in  : \diff_c^\infty (  \T \times M ) :  X \text{ has null $ \T$-coordinate} \rbrace  \; . \]
% Note that every elements of $\sg_3$ is of the form:
% \begin{equation*} (\theta,y) \in \T \times M  \mapsto   \left(0, \sum_{1 \le i \le N } \phi_i ( \theta ) \cdot f_i(y)\right) \in \R \times TM  \; ,\quad \text{for } \phi_i \in C^\infty ( \T , \R )\qand  f_i\in \diff_c^\infty (  M) \; .\end{equation*}  
The non-trivial remaining part of the proof is to show the following:
\begin{proposition} \label{p.hatg2}
Any vector field  of $\sg_3  $ has pluggable flows: 
  \begin{equation*}  \sg_3 \subset  \spp \; . \end{equation*}  
\end{proposition} 
The proof of this proposition will occupy \cref{Eigen,proof of deceig}.  In the next section we show that it implies main \cref{mainplug}. The proof of several propositions will involve the following notation. If $ \sg $ and $ \sh $  are two sub-Lie algebras, we denote $[ \sg , \sh  ]$ the vector space spanned by Lie brackets of elements of $\sg$ and $\sh$: 
\begin{equation*}  [ \sg , \sh  ]  := \left\{  \sum_{ \text{i:finite}  } [ X_i , Y_i ] : X_i \in \sg , \ Y_i \in \sh  \right\} \; . \end{equation*}  
 \paragraph{Parametric counterpart.} 
 Similarly we set:
\begin{equation*}  \sg_{3\cP}  := \lbrace ( X_p)_{ p \in \cP} \in \diff_c^\infty ( \T \times M) _\cP :  X_p \in \sg_3 , \  \forall p \in \cP  \rbrace \end{equation*} 
and 
\begin{equation*}  \widehat {\sg_{3\cP}}:=   \lbrace \widehat{X_\cP} : X_\cP  \in \sg_{3\cP} \rbrace \subsetneq \sg_3 ( M \times \cP)   \; .\end{equation*} 
 We will prove the following parametric counterpart of \cref{p.hatg2} in \cref{Eigen,proof of deceig}:
\begin{proposition} \label{p.hatg2.para}
Every smooth family of vector fields  in $ \sg_3$ has a $\cP$-pluggable flow: 
\begin{equation*}   \sg_{3\cP} \subset \spp_\cP \; .\end{equation*} 
\end{proposition} 
\subsection{Proof of main \cref{mainplug} and \cref{mainplug cP} } \label{subsec proof of main}
The first step of the proof is the following:
\begin{proposition} \label{p.mainspp}
Any vector field has a pluggable flow:  
 \begin{equation*}  \spp = \diff_{c}^\infty ( \T  \times M) \; . \end{equation*}  
\end{proposition}
\begin{proof}
As $ \spp $ is a Lie algebra by \cref{p.sppclose} and since $ \sg_1 $ and $ \sg_3$ are in $ \spp $ by \cref{sub12} and \cref{p.hatg2}, it suffices to show that the Lie algebra generated by $\sg_1$ and $\sg_3$ equals $\diff_{c}^\infty ( \T  \times M)$. 
We first prove the statement of the proposition for $M=\R^n$.  
Let $ {X \in \diff_c^\infty ( \T \times M)} $ and  $f\in C^\infty_{c}(\T\times \R^n, \R)$ be its $ \T$-coordinate. Let $W \in \sg_3$ be: 
\begin{equation} \label{Xmainspp}
    W :=(0,f,0,\ldots,0) \; .
\end{equation} 
Let $\rho\in C^\infty(\T\times M,\R)$ be a compactly supported function which is equal to $1$ near the support of $f$   and does not depend on the $\T$ coordinate  . Let $ Y \in \sg_1 $ be:  
\begin{equation} \label{Ymainspp}
    Y(\theta,y_1,\ldots, y_n)=(\rho(\theta, y)\cdot y_1,0,\ldots,0) \; . \end{equation}
A computation gives a Lie bracket of the form:
\begin{equation} \label{crochetmainspp}
    [W,Y](\theta,y)=(f(\theta,y),- \partial_{\theta} f(\theta,y) \cdot y_1,0, \ldots, 0) \; .
\end{equation}
Thus $ X - [ W, Y] $ is in $ \sg_3 $  which gives the desired result for $M=\R^n$. 

When $M$ is another manifold, we fix a locally finite covering $(U_i)_i$  by balls $U_i$. Using a partition of the unity, every $X\in \diff_{c}^\infty ( \T  \times M)$ can be written as a sum:
\begin{equation} \label{partitionmainspp}
    X= \sum_i X_i \; . 
\end{equation}
where each $X_i$ is supported by $U_i$. As $X$ is compactly supported, the $X_i$ are almost all null and thus the above sum is finite.  As each $U_i$ is diffeomorphic to $\R^n$, we can apply the case $M=\R^n$ which gives $Y_i,Z_i\in \sg_3$ and $W_i\in \sg_1$  all supported by $ \T \times U_i $ such that $[W_i,Y_i]+Z_i=X_i$. We conclude by summing over $i$.  
\end{proof}
 
\begin{remark}\label{rema:mainspp} We proved that   $\left[\sg_3  ,\sg_1\right]  +\sg_3  =\diff_c^\infty(\T\times M)$. \end{remark} 
We now have the tools to show: 

\begin{proof}[Proof of \cref{mainplug}]
Let $G\in \Diff_c^\infty ( \T \times M ) $ and let $(G_t)_{ t \in [ 0,1] } $ be a compactly supported smooth path from $ G_0 = \id$ to $G_1 = G$ in $\Diff_c^\infty(\T \times M)$.   %Let $X = ( \partial_t G_t)_{ t \in [ 0,1]} $.
Derivatives $ X_t := \partial_t G_t \circ G_t^{-1}$  define a smooth family $X = ( X_t)_t$ of vector fields all supported in a compact subset $K\subset V\setminus \partial V$.  In particular,   the time $1/N$-map $ F_i$ of the vector field $X_{i/N}$ is supported by $K$, and likewise for $F=  F_{N-1}\circ \cdots \circ  F_0$. 
By definition of the Whitney topology, it suffices to show that for every $r\ge 1$, when $N$ is large,  the map  $F$ is $C^r$-close to $G$  to obtain that $F$ is close to $G$ in $\Diff_c^\infty(\T \times M)$. 

Note that indeed each $ F_i$ is $O_{C^r}(1/N^2)$-close to $\tilde F_{i}:=G_{(i+1)/N }\circ G_{i/N} ^{-1}$ and so it holds: 
\begin{equation} G=  \tilde F_{N-1}\circ \cdots \circ  \tilde F_0 =  F_{N-1}\circ \cdots \circ F_0+O_{C^r} (1/N)=  F + O_{C^r} (1/N) \; .\end{equation} 
As each $ F_{i}$ belongs to $\P $, and since $\P$ is a closed group by \cref{c.closure cP}, it comes that
$F$  belongs to $\P$ and 
 its limit $G$ when $ N  \to \infty$ as well. 
\end{proof}

\paragraph{Parametric counterpart.} To prove  \cref{mainplug cP} we will use  the following parametric counterpart of \cref{p.mainspp}: 
\begin{proposition} \label{p.mainsppcP}
The flow of every $\cP$-family of vector fields is  $ \cP$-pluggable: 
 \begin{equation}    \spp_\cP = \diff_c^\infty ( \T \times M ) _\cP \; . \end{equation}  
\end{proposition}
\begin{proof}
As in the proof of \cref{p.mainspp}, it suffices to show that the Lie algebra generated by $ \sg_{1\cP}$ and $   \sg_{3\cP}    $  equals $\spp_\cP$.
We start with the case $ M = \R^n $ and $ \cP = \R^d$. 
Let $ X_\cP = ( X_p ) _{ p \in \cP} $ be a $\cP$-family  in $ \diff_c^\infty ( \T \times M ) _\cP $ and $f _p \in \diff_c( \T \times M )  $  be the $\T$-coordinate of $ X_p$ for each $p \in \cP$. Let $W_\cP \in \sg_{3\cP} $ be:
\[  W_\cP :=  (( 0, f_p  , 0 , \dots , 0))_{p \in \cP}  \; .  \]  
With $Y$ as defined in \cref{Ymainspp}, it holds by \cref{crochetmainspp}: 
\[ [ W_p  , Y  ] ( \theta , y)  = ( f_p( \theta , y)  , - \partial_{\theta} f_p ( \theta , y)  \cdot y_1 , 0 , \dots , 0 ) )_{ p \in \cP}  \; .      \] 
We set $ Y_\cP := ( Y )_{p \in \cP  } \sg_{1\cP} $. The latter computation gives:
\begin{equation} [ W_\cP , Y_\cP ] ( \theta , y) =  (( f_p ( \theta ,y)  , - \partial_{\theta} f_p (\theta,y)  \cdot y_1  , 0 , \dots , 0 ) ) _{ p \in \cP}  \; . \end{equation}
Thus $X_\cP -  [ W_\cP,  Y_\cP ]  $ is in $ \sg_{3\cP }$. 
Now for the general case we conclude by using a partition of unity as before. 
\end{proof}
It allows to conclude: 
\begin{proof}[Proof of \cref{mainplug cP}]
Consider $  G_\cP  \in  \diff_c^\infty ( \T \times M)_\cP $. As a corollary of \cref{mainplug}, for any $r \ge 0$, the map $ \widehat{ G_\cP}  \in \Diff_c^\infty ( \T \times M \times \cP)  $ can be approximated by a composition of flows of vector fields with a zero $ \cP$-coordinate.
Thus the family $ G_\cP$ can be approximated arbitrarily close by a composition of compactly supported families of flows of vector fields in the $ C^r$-norm for any $r \ge 0$. By  \cref{p.mainspp}, each of the families of flows belongs to $\P_\cP$. Since $\P_\cP$ is a  group by \cref{c.closure cP} the above composition belongs to $ \P_\cP $ as well and  it comes that $(G_p)_{ p \in \cP}$ is in $\P_\cP$ by closedness of $ \P_\cP$. \end{proof}
\subsection{Eigenvectors of the adjoint representation}\label{Eigen}
\label{section proof of mainplug}

Observe that  the Lie algebra generated by $\mathfrak{g}_1$ and $\mathfrak{g}_2$  contains the vector fields that do not depend on the $ \T$-coordinate.  To obtain more elements in $ \spp$ we study the eigenvectors of operators of the form: 
\begin{equation*} \ad_X: Y\in \diff_c^\infty ( \T \times M) \mapsto [ X, Y ] \in \diff_c^\infty ( \T \times M)\; ,\end{equation*} 
for $X\in \diff_c^\infty ( \T \times M)$. 

In this subsection, we show that if  $Y$ is an eigenvector of $\ad_X$ for $X\in \spp$  then  $Y$ is in $\spp$ (see \cref{p.adjunctionstability}). This will enable us to prove \cref{p.hatg2} in \cref{proof of deceig}. 
\begin{definition}
Let $V$ be a manifold and let $ \sg   $  be a subspace of  $ \diff_c^\infty ( V ) $. We denote:
\begin{equation*}  \seig (  \ad_\sg ):= \lbrace  Y  \in \diff_c^\infty ( V ) : \exists X \in \sg \text{ such that } Y = [ X, Y ]    \rbrace    \; .  \end{equation*}  
\end{definition}
The following is the key proposition enabling to construct new examples of pluggable flows: 
\begin{proposition}\label{p.adjunctionstability}If  $Y\in \diff_c^\infty ( \T \times M)$ satisfies $\ad_X Y=Y$  for $X\in \spp$  then  $Y$ is in $\spp$:   
\begin{equation*}  \seig ( \ad_\spp ) \subset  \spp \; .   \end{equation*}  
\end{proposition}
\begin{proof}
   Given a map $f \in \Diff^\infty_c ( M) $ and a vector field $ W  \in \diff_c^\infty ( M ) $ we denote the pushforward of $W$ by $f$ as:
\begin{equation*}
    \Ad_f W := f_* W := Df \circ W \circ f^{-1} \;. 
\end{equation*}
This notation is consistent with the usual composition rules given by the following commuting diagram:
\begin{equation*}  
\begin{array}{rcccl} 
	&					&{W}		&					&\\
	&M	&\longrightarrow	&TM&\\
f	&\downarrow			&				&\downarrow			&Df\\
	&M &\longrightarrow	&TM&\\
	&					&{\Ad_f W}		&					&\\
\end{array} \; . \end{equation*} 
The following contains a  key idea for the proof of the main theorem: 
\begin{lemma} \label{horousefull}If $Y$ satisfies $[X,Y]=Y$, then for every $s,t  \in \R$  it holds:
\begin{equation}  \Fl^t_Y  = \Fl_X^{-s}  \circ  \Fl_Y^{ t \cdot e^{-s}  }  \circ \Fl_X^{s} \; .  \end{equation} 
\end{lemma}
\begin{proof}
We recall the following well known result on adjunction of vector fields by flows, see e.g. \cite[Prop 1.9]{kobayashi1963foundations}:
\begin{fact}
Let $ U,W \in \diff_c^\infty ( M) $ be two vector fields then it holds: 
\begin{equation*}
    \partial_t \Ad_{\Fl_U^t} W_{|t=0} = - [ U, W ] \;. 
\end{equation*}
\end{fact}
First observe that by the latter fact it holds  $ X = \Ad_{\Fl_X^{s}}( X)$ for every $ s \in \R$. 
Let then  $Y_s:= \Ad_{\Fl_X^{s}}( Y) = D \Fl_X^{s} \circ Y \circ \Fl_X^{-s}$.
Observe that since $ Y = [ X, Y]$, it holds:
\begin{equation}  [ X, Y_s ] = [ \Ad_{ \Fl_X^s} ( X ) , \Ad_{ \Fl_X^s} ( Y  ) ] = \Ad_{ \Fl_X^s} ( [X,Y] ) = \Ad_{ \Fl_X^s} ( Y  ) = Y_s \; .   \end{equation} 
Also for every $ s \in \R$, we have:
\begin{equation} \partial_s Y_s = \partial_t (Y _{s+t} ) _{|t=0}=
\partial_t \Ad_{\Fl_X^{t}}( Y_s)_{|t=0}  \; .\end{equation}
Thus by the latter fact it follows: 
\begin{equation}  \partial_s Y_s = - [X, Y_s ] = -Y_s \; . \end{equation} 
Consequently $ \partial_s Y_s =  - Y_s $ and thus $  e^{-s} \cdot Y =  Y_s = \Ad_{\Fl_X^{s}}(  Y)  $. After integration between $0$ and $t$, we obtain:
\begin{equation}
    \Fl_{ e^{-s} \cdot Y}^t = \Fl_X^s \circ \Fl_Y^t \circ \Fl_X^{-s} \; .  
\end{equation}
As $\Fl_{ e^{-s} \cdot Y}^t  = \Fl_{Y}^{t \cdot e^{-s}}   $ we obtain the desired result by composing the latter equation on the right by $\Fl_X^s $ and on the left by $\Fl_X^{-s}$.  
\end{proof} 
We can now prove \cref{p.adjunctionstability}.
Fix a neighborhood $\cN  $ of $\id \in \Diff^\infty ( \T \times M ) $ and let us show the existence of a plugin in $\cN$ whose output is $\Fl^t _Y$. Let $\cN_c $ be the neighborhood of $\id  \in \Diff_c^\infty ( \T \times M )$  given by \cref{l.neigh.of.Id}. Let   $s \in \R $ be sufficiently large such that $\Fl^{e^{-s}\cdot t}_Y$ belongs to the neighborhood $\cN_c $ of the identity. Then for any $k\geq 1$ large enough $\mathrm{Fl}^{e^{-s}\cdot t}_Y$ is the output of a plugin $g\in \cN$ of step $2^{-k}$. Since $ X \in \spp$, it holds $\Fl_X^{ s}\in \P$ so for any $k\geq 1$ large enough $\Fl_X^{ s}$ and $\Fl_X^{-s}$ are  the outputs of plugins $h$ and $f$ in $\cN$ of respective steps $2^{-k}$ and $2^{-k-1}$. So for any large $k$ the map $\mathrm{Fl}^{t}_{Y}$ is the output of the plugin $(h\star g)\star f$ with step $2^{-k-2}$.  Since this holds for any neighbourhood $\cN$ of the identity the plugin $(h\star g)\star f$  can be taken arbitrarily close to identity. Therefore  for every $t$, there exists a plugin with output   $\Fl_Y^t $   arbitrarily close to identity for any small enough step. Hence    $ Y $   is in $\spp$.  
\end{proof}    
  
A second main ingredient of the proof of the main theorem is the following of independent interest:
\begin{proposition}\label{p.pre_deceig} For every $T\in \diff^\infty_c(V)$, there exist finite families $(X_i)_i,(Y_i)_i,(Z_i)_i$ of vector fields in $  \diff^\infty_c(V)$ such that: 
\begin{equation*}  T=\sum_i [Y_i,Z_i]\qand Y_i=[X_i,Y_i]\; .\end{equation*} 
\end{proposition} 
The proof of this proposition and its parametric counterpart will occupy the full  \cref{{proof of deceig}}. 
Now note that the following is an isomorphism of Lie algebras:
\begin{equation*}   \si_2 : X \in \diff_c^\infty ( M ) \mapsto ( 0 , X) \in \sg_2 \; .  \end{equation*} 
Thus, by applying this isomorphism to the image of sets involved in the statement of  \cref{p.pre_deceig} for every $T\in \sg_2$, there exist finite families $(X_i)_i,(Y_i)_i,(Z_i)_i$ of vector fields in $\sg_2$  such that: 
\begin{equation*}  T=\sum_i [Y_i,Z_i]\qand Y_i=[X_i,Y_i]\; .\end{equation*} 
In other words, we proved:
 \begin{coro} \label{p.deceig}Any element of $ \sg_{2}$ can be written as a sum of Lie brackets of  elements  of $\sg_{2}$  with elements in $\seig( \ad_{\sg_{2} } )\cap \sg_{2 }$.
 \end{coro}
This corollary allows to deduce: 
\begin{proof}[Proof of \cref{p.hatg2}]
Let $p_\theta: \T\times M\to \T$ be the projection on the  $ \T$-coordinate.

First note that by Fourier decomposition theorem, the space $\sg_3$ is the closure of the vector space spanned by elements of the form:
% \begin{equation*} (\theta,y) \in \T \times M  \mapsto   \left(0,\phi ( \theta ) \cdot Y(y)\right) \in \R \times TM  \; ,\quad \text{for } \phi  \in C^\infty ( \T , \R )\qand  
% Y \in \diff_c^\infty (  M) \; .\end{equation*} 
\[  \phi \circ p_\theta  \cdot  Y \quad \text{for } \phi  \in C^\infty ( \T , \R )\qand  Y \in \sg_2 \; . \]
Since $ \spp$ is a closed vector space by \cref{p.sppclose}, it suffices to show any such $ \phi \circ p_\theta \cdot Y $ is in $\spp$. 
% Consider such a  $ Y \in  \sg_2 $ and $ \phi \in C^\infty ( \T ,  \R) $. It suffices then to show that $ \phi \circ p_\theta\cdot Y$ is in $\spp$.
To do so, we first start with the case where there exists $ X \in \sg_2 $ such that $ Y = [ X, Y ] $, i.e. we  assume that $  Y \in \seig( \ad_{\sg_2 }  ) $.
Since the $ \T$-coordinate $p_\theta \circ X$   of $X$ is zero, it follows:
\begin{equation}\label{phiYeig} [ X , \phi \circ p_\theta \cdot Y  ] =\phi \circ p_\theta \cdot DY(X) -
DX  ( \phi \circ p_\theta \cdot Y) = \phi \circ p_\theta \cdot [ X ,  Y  ]= \phi \circ p_\theta \cdot Y   \; . \end{equation} 
Thus by \cref{p.adjunctionstability}   we have $ \phi \circ p_\theta \cdot Y \in \spp $ since $X\in \sg_2 \subset \spp$. 
Now in the general case, for every $Y\in \sg_2$,  by  \cref{p.deceig}, there exist an integer $ N \ge 1 $, $  Z _i \in \sg_2$ and $ Y _i \in \seig  ( \ad_{\sg_2 } ) \cap \sg_2  $ for any $ 1 \le i \le N$, such that:
\begin{equation} \label{Ysum}   Y =  \sum_{ 1  \le N } [Y_i , Z_i ]  \; .  \end{equation}  
By the first case for any $ 1 \le i \le N$ we have  $ \phi  \circ p_\theta\cdot Y_i \in \spp $. Since $ \spp $ is a Lie-algebra and each $Z_i$ is in $ \sg_2 \subset \spp$, it immediately follows:
\begin{equation}\label{dec g3}  \phi\circ p_\theta \cdot Y = \sum_{ 1 \le i \le N} \phi \circ p_\theta \cdot [ Y_i , Z_i ] =  \sum_{ 1 \le i \le N}  [ \phi\circ p_\theta \cdot Y_i , Z_i ]    \in  \spp  \; .  \end{equation}
\end{proof}
\paragraph{Parametric counterpart.}The following is the parametric counterpart of \cref{p.adjunctionstability}.
% To prove \cref{p.hatg2.para}, we will use:
\begin{proposition}\label{p.adjunctionstabilitycP}
If $ Y_\cP \in \diff_c^\infty ( \T \times M)_\cP  $ satisfies $ [ X_\cP , Y_\cP ] = Y_\cP $ for $ X_\cP \in \spp_\cP $, then $ Y_\cP \in \spp_\cP $.
\end{proposition}
  \begin{proof}
Let  $X_\cP=:( X_p)_{p \in \cP } $ and $Y_\cP=:( Y_p)_{p \in \cP } $. 
By \cref{horousefull}, it holds: 
\begin{equation}  \Fl_{Y_p}^t = \Fl_{X_p}^{-s} \circ \Fl_{Y_p}^{t \cdot e^{-s} } \circ \Fl_{ X_p}^s \; ,  \end{equation} 
for any $t,s \in \R $. 
For any neighborhood $ \cN $ of $  \id \in \Diff^\infty ( \T  \times M)_\cP $, we denote  $ \cN_{ c } $  the neighborhood of $  \id \in \Diff_c^\infty ( \T  \times M)_\cP $ given by \cref{l.neigh.of.Id cP}. 
For all $t \in \R $ and for $s$ large enough $ ( \Fl_{Y_p}^{t \cdot e^{-s} }  )_{ p \in \cP} \subset \cN_{c } $ and thus it is the output of a $\cP$-plugin in $ \cN$ of any small step.
We now regard the $\star$ product of the latter plugins with $\cP$-plugins of $ (\Fl^s_{X_p} )_{ p \in \cP} $ and $ ( \Fl^{-s}_{X_p} )_{ p \in \cP}$ to obtain  a  $\cP$-plugin with output $ ( \Fl_{Y_p}^{t}  )_{ p \in \cP} $ of any small step. Moreover, by construction, this $\cP$-plugin can be taken arbitrarily close to identity family, which ends the proof.  
 \end{proof}
 
For $\sg_\cP \subset \diff_c^\infty( V)_\cP $, we   denote:
\begin{equation*}  \seig (  \ad_{\sg_\cP}  ):= \lbrace  Y_\cP  \in \diff_c^\infty ( V )_\cP : \exists X_\cP \in \sg_\cP \text{ such that } Y_\cP = [ X_\cP , Y_\cP ]    \rbrace \; . 
\end{equation*} 
The following parametric counterpart of \cref{p.pre_deceig} holds:
\begin{proposition} \label{p.deceigcP}
For every $T_\cP \in \diff^\infty_c(V)_\cP$, there exist finite families $(X_{i\cP})_i,(Y_{i\cP})_i,(Z_{i\cP})_i$ of vector fields in $  \diff^\infty_c(V)_\cP$ such that: 
\begin{equation*}  T_\cP =\sum_i [Y_{i\cP},Z_{i\cP}]\qand Y_{i\cP}=[X_{i\cP},Y_{i\cP}]\; .\end{equation*} 
\end{proposition}
Using the isomorphism:
\[ X_\cP \in \diff_c^\infty ( M)_\cP  \mapsto  ( 0, X_p )_{p \in \cP}  \in \sg_{2 \cP}   \] 
leads as before to: 
 \begin{coro} \label{c.deceigcP} Any element of $ \sg_{2\cP}$ can be written as a sum of Lie brackets of  elements  of $\sg_{2\cP}$  with elements in $\seig( \ad_{\sg_{2\cP} } )\cap \sg_{2\cP}$.
 \end{coro}
The latter allows to deduce:
\begin{proof}[Proof of \cref{p.hatg2.para}]
By the Fourier decomposition theorem, the space $ \sg_{3 \cP} $ is the closure in $ \diff_c^\infty ( \T \times M ) $ of the vector space spanned by vector fields of the form $( \phi \circ p_\theta \cdot Y_p)_{ p \in \cP} $ with $ \phi \in C^\infty( \T , \R) $ and $ (Y_p)_{p \in \cP} \in \sg_{2\cP} $.
% Consider $ Y_\cP \in  \sg_2 $ and $ \phi \in C^\infty ( \T ,  \R) $. 
Since $ \spp_\cP$  is a closed vector space by \cref{p.sppclose.para}, it suffices then to show that $(\phi \circ p_\theta\cdot Y_p ) _{ p \in \cP} $ is in $\spp_\cP$. 
To do so, we first start with the case where there exists $ X_\cP \in \sg_{2\cP}  $ such that $ Y_\cP = [ X_\cP, Y_\cP  ] $.  By \cref{phiYeig,dec g3} of the proof of  \cref{p.hatg2}, for each $p \in \cP$, it holds $ \phi \circ p_\theta \cdot Y_p = [ X_p ,\phi \circ p_\theta \cdot Y_p ]  $. And thus the family $ ( \phi \circ p_\theta \cdot Y_p ) _{p \in \cP} $ is in $ \spp_\cP$ by \cref{p.adjunctionstabilitycP}. 

Now in the general, by the latter \cref{c.deceigcP} we can decompose:
\begin{equation}
    Y_\cP = \sum  [ Y_{i\cP} , Z_{i\cP} ] \; ,
\end{equation}    
with $ Y_{i\cP} \in \seig ( \sg_{2\cP} )\cap  \sg_{2\cP}$ and $Z_{i\cP} \in \sg_{2\cP} $.
Thus \cref{dec g3} applied for each $p \in \cP$ leads to: 
\begin{equation} ( \phi \circ p_\theta \cdot Y_p )_{p \in \cP} =\sum  [ ( \phi \circ p_\theta \cdot Y_{i p}  )_{p \in \cP},  Z_{i\cP}]  \;,  \end{equation}
where $ ( Y_{ip} )_{p \in \cP}  = Y_{i\cP}$. 
Since $Z_{i\cP} \in \sg_{ 2\cP} \subset \spp_\cP $ and $ ( \phi \circ p_\theta \cdot Y_{i p}  )_{p \in \cP} \in  \spp_\cP $ by the first case, it follows that $ ( \phi \circ p_\theta \cdot Y_p )_{p \in \cP}$ lies in the sub Lie algebra $ \spp_\cP$ as well.
\end{proof}

\subsection{Decomposition of  vector fields}
\label{proof of deceig}
To prove the \cref{mainplug},  it remains only to show \cref{p.pre_deceig}.  We first prove this proposition in the case $M=\R^n$, which we will deduce from the case $M=\R$ and a parametric version thereof.
This whole section is dedicated to this proof. 
  For a manifold  $M$, let us denote:
  \begin{equation*}  \seig( M) := \{Y\in \diff_c^\infty (  M  ): \exists X\in\diff_c^\infty (  M  ) \text{ such that }\ad_XY=[X,Y]=Y\}\; .\end{equation*} 
 We can rephrase  \cref{p.pre_deceig} as: 
\begin{proposition}\label{egatiliteHN} 
Every vector field in $T\in \diff^\infty_c(V)$ is a finite sum of vector fields of the form  $[Y_i,Z_i]$ with 
$Y_i\in  \seig( M)$. In other words:
\begin{equation*}  [\seig(M),\diff^\infty_c(M)]=\diff^\infty_c(M).\end{equation*}  
\end{proposition}   
Two key ingredients of the proof of this proposition are the following observations:
\begin{fact}\label{horo historique}The vector field  $Y : y \mapsto  1 $ on $\R$ satisfies $[X,Y]= Y$ with $X: y\in \R\mapsto -y$.
\end{fact} 
\begin{fact}\label{ad horo stab} 
For any diffeomorphism $\psi\colon W\to V$ between manifolds and any vector fields ${X,Y\in \diff_c^\infty(V)}$  satisfying  
$[X,Y]= Y$, it holds $[\psi^*X,\psi^*Y]= \psi^*Y$.
 \end{fact}
 An important  consequence of the latter fact is:
 \begin{fact}\label{ad eig stab} 
For any diffeomorphism $\psi\colon W\to V$ from a manifold $ W$ into $V$, if $X\in  [\seig(V),\diff^\infty_c(V)]$, then $ \psi^* X\in [\seig(W),\diff^\infty_c(W)]$. 
 \end{fact}
  
 Let $\cQ$ be a manifold. We define: 
\begin{equation*} \seig(\R)_\cQ=\Bigl\{Y_\cQ\in \diff^\infty_c(\R)_\cQ:  \text{ there exists a family } X_\cQ\in \diff^\infty_c(\R)_\cQ\text{ such that } 
[X_\cQ,Y_\cQ]=Y_\cQ \Bigr\}.\end{equation*} 
We can  prove   \cref{egatiliteHN} in  the case $M = \R$:
\begin{lemma}\label{case cc dim 1N} We have:
 \begin{equation*}   [ \seig( \R)  , \diff_c^\infty( \R )  ]= \diff_c^\infty( \R )  \; .   \end{equation*} 
 Moreover for any manifold 
 $\cQ$,   every  $T_\cQ \in \diff_c^\infty(\R)_{\cQ}$ satisfies $T_\cQ=[Y_\cQ,Z_\cQ]$ for some ${Y_\cQ\in \seig(\R)_{\cQ}}$ and $Z_\cQ\in \diff_c^\infty(\R)_{\cQ} $. 
\end{lemma}
\begin{proof}
We first give an intuitive idea of the proof using
\cref{horo historique}. First note that for every $T\in \diff_c^\infty(\R)$, there exist $X,Y,Z\in \diff^\infty(\R)$ such that $Y=[X,Y]$ and $T=[Y,Z]$. Indeed it suffices to take  $X\colon y\mapsto -y$,  $Y\colon y\mapsto 1$ and   $Z\colon y\mapsto \int_{-\infty}^yT(t)\,  dt$. This almost proves the first assertion of the Lemma.
To have exactly the desired result we shall modify $X$, $Y$ and $Z$ to make them compactly supported.
Let $\widetilde{T}\in \diff^\infty_c(\R)$. Take intervals $[-A,A]\subset (-a,a)$ containing its support. Consider the map $\psi\colon y\mapsto y\cdot e^{e^{(a^2-y^2)^{-1}}}$ from $(-a,a)$ to $\R$. 
We compute its derivative at $ y \in (-a,a)$ by:
\begin{equation}D \psi(y)=  \phi (  y) \cdot e^{e^{(a^2-y^2)^{-1}}}\; ,\quad \text{ where }\phi(y) := 1+2y^2\cdot(a^2-y^2)^{-2}e^{(a^2-y^2)^{-1}} \; .\end{equation}
Thus the map $ \psi $ is a diffeomorphism. This allows to consider the push-forward:
\begin{equation} T=\psi_*\widetilde{T}= D\psi\circ\widetilde T\circ \psi^{-1}
\; .\end{equation} 
Now we define as above:
 \begin{equation} X\colon y\mapsto -y\, , \quad Y\colon y\mapsto 1\qand Z\colon y\mapsto \int_{-\infty}^yT(t)\,  dt\, .\end{equation} 
Note that $\tilde T= \psi^*T$ and put $\widetilde{X}= \psi^*X$ and $\widetilde{Y}= \psi^*Y$. Then  by \cref{ad horo stab}, it holds:
\begin{equation}   \widetilde{Y} = [\widetilde{X}, \widetilde{Y}]\qand [\widetilde{Y},\widetilde{Z}]=\widetilde{T}\; .\end{equation} 
Let us show that these pulled-back vector fields on $(-a,a)$ extend smoothly by $0$ to $\R$.
For any vector field $S \in \diff_c^\infty(\R)$, we have $\psi^*S=(D\psi)^{-1} \circ S\circ \psi = \tfrac{ S \circ \psi }{D\psi } $. 
It follows:
\begin{equation}  \widetilde{X} (y) = \psi^*X (y)  = \tfrac{- \psi (y) }{ D \psi (y)  } =  \tfrac{ -y }{\phi( y) }\qand\widetilde{Y} (y) = \psi^*Y(y) =\tfrac{1 }{D \psi (y) },\end{equation} 
Observe that $ \phi (y) $ grows exponentially fast to $ + \infty $ when $ \vert y \vert \to a$. Thus $\widetilde{X}$   and $\widetilde{Y}$ have all their derivatives tending to $0$ as $y$ tends to $\pm a$. So they extend smoothly by $0$ to vector fields in $\diff_c^\infty(\R)$. 
Also  note that $Z(y)=0$ for $y\leq \psi(-A)$ and $Z(y)= \int_\R T(t)\, dt$ for  $y\geq \psi(A)$. Thus:
\begin{equation} \widetilde{Z}(y)= \psi^*Z(y) =0\quad \mbox{ for } -a  < y\leq -A\qand \widetilde{Z}(y) =\frac1{D \psi(y)}\cdot \int_\R T(t)\cdot dt\quad \mbox{ for } A \leq y <   a .\end{equation} 
Likewise $\widetilde{Z}$ extends smoothly by  $0$ to form a  vector field in $\diff_c^\infty(\R)$. As  $ \widetilde{Y} = [\widetilde{X} , \widetilde{Y}  ] $ and $[\widetilde{Y},\widetilde{Z}]=\widetilde{T}$, this proves the first assertion of the lemma.
\medskip

For the parametric assertion, observe that $\psi$, $\widetilde{X}$ and $\widetilde{Y}$ depends only on the segment $[-a,a]$. Hence given $ \widetilde{T}_\cQ=(\widetilde{T}_q)_{q\in\cQ}\in \diff_c^\infty(\R)_{\cQ}$, we set $[-A,A]\subset (-a,a)$ containing the supports of all $\widetilde{T}_q$ and define  $\psi$, $\widetilde{X}$ and $\widetilde{Y}$ as above. Then we observe that $\widetilde{Z}_q= \psi^*\int T_q$ depends smoothly on $q$ and define a family $\widetilde{Z}_\cQ\in \diff_c^\infty(\R)_\cQ$ which satisfies the desired equalities with $ \widetilde{T}_\cQ$ and the constant families   $(\widetilde{X})_{q\in \cQ}$ and $(\widetilde{Y})_{q\in \cQ}$. 
\end{proof}
We are going to  use the parametric assertion of the latter lemma to obtain:
\begin{lemma}\label{case cc dim N} 
We have $[ \seig(\R^n),\diff^\infty_c(\R^n)]=\diff^\infty_c(\R^n)$.
\end{lemma}
\begin{proof}
\cref{case cc dim 1N} corresponds to the case $n=1$.  For $n\ge 2$, given $T\in \diff^\infty_c(\R^{n})$, we write its components as $T= (T_1,\dots, T_n)$. By linearity of the condition, it suffices to show that each vector field $( 0, \dots, 0, T_i, 0, \dots, 0)$ is in $[\seig(\R^n),\diff^\infty_c(\R^n)]$. 
Using an adjunction by a permutation of the coordinates and \cref{ad eig stab}, its suffices to show that each  $( T_i, 0, \dots, 0)$ is in $[\seig(\R^n),\diff^\infty_c(\R^n)]$. In other words, it suffices to prove that the following subalgebra $\sh(\R^n)$ of $\diff^\infty_c(\R^n)$ is included in $[\seig(\R^n),\diff^\infty_c(\R^n)]$ :
\[\sh(\R^n) := \{y\in \R^n\mapsto (h(y),0,\dots,0): h\in C^\infty_c(\R^n,\R)\}\; .
\]
To this end, note that $X_{\R^{n-1}}\in \diff^\infty_c(\R)_{\R^{n-1}}\mapsto \widehat {X_{\R^{n-1}}}\in \sh(\R^n)$ is an  isomorphism of Lie algebras. By \cref{case cc dim 1N} with $\cQ=\R^{n-1}$, we have:
\[  \sh(\R^n)=   [\widehat{\seig(\R)_{\cQ}} , \sh(\R^n)]\; . 
\]
Finally we note that $\widehat{\seig(\R)_{\cQ}}$ is formed by vector fields of the form $\widehat{Y_\cQ}$ such that
$Y_\cQ= [X_\cQ,Y_\cQ]$ for $X_\cQ\in \diff_c^\infty(\R)_\cQ$. Thus $\widehat{Y_\cQ}= [\widehat{X_\cQ},\widehat{Y_\cQ}]$, this proves that $\sh(\R^n)\subset   [\seig(\R^n)  , \sh(\R^n)]$.  

\end{proof}
We can now treat the general case:
 \begin{proof}[Proof of  \cref{egatiliteHN}]
Let $T\in  \diff^\infty_c(M)$. Then it decomposes in a finite sum $T=\sum_i T_i$ where each $T_i$ is compactly supported in an open set $U_i$ which is diffeomorphic to $\R^n$ via a map $\psi_i\colon U_i\to \R^n$. Consider the push forward $\psi_{i*}T_i\in   \diff^\infty_c(\R^n)$. By \cref{case cc dim N} 
the field  $\psi_{i*}T_i|U_i$   belongs to  $[ \seig(\R^n),\diff^\infty_c(\R^n)]$. Thus by \cref{case cc dim 1N}, it holds $T_i|U_i\in [ \seig(U_i),\diff^\infty_c(U_i)]$.  This means that   $T_i=\sum_{\text{Finite}} [Y_j,Z_j]$ with $Y_j=[X_j,Y_j]$ for some $X_j, Y_j,Z_j\in \diff^\infty_c (\R^n)$. Extending all these vector fields by $0$, we obtain that $T_i$ belongs to $ [ \seig(M),\diff^\infty_c(M)]$.   So does $T=\sum_i T_i$. 
 \end{proof}
\paragraph{Parametric counterpart.}
We set:
\[ \seig(M)_\cP := \lbrace Y_\cP \in \diff_c^\infty(  M )_\cP :   Y_\cP = [ X_\cP , Y _\cP 
] \text{ with } X_\cP \in  \diff_c^\infty(  M )_\cP \rbrace \;.  \]
We shall prove \cref{p.deceigcP} that we  rephrase as: 
\begin{proposition}\label{p.deceigcP rephrase} 
It holds: 
\begin{equation*}  [\seig(M)_\cP ,\diff^\infty_c(M)_\cP]=\diff^\infty_c(M)_\cP.\end{equation*}  
\end{proposition} 
Similarly to the proof of \cref{egatiliteHN}, \cref{p.deceigcP rephrase}  is an easy consequence of the following:  
\begin{lemma}\label{case cc dim Np} We have $[\seig(\R^n)_\cP,\diff^\infty_c(\R^n)_\cP]=\diff^\infty_c(\R^n)_\cP$. \end{lemma}
\begin{proof}
As in the proof of \cref{egatiliteHN}, we only need to prove that  the following Lie subalgebra $\sh(\R^n)_\cP$ of $\diff^\infty_c(\R^n)_\cP$ is included in $[\seig(\R^n)_\cP,\diff^\infty_c(\R^n)_\cP]$ :
\[\sh(\R^n)_\cP := \{  ((h_p,0,\dots,0))_{p\in \cP}: h_\cP=(h_p)_{p\in \cP} \in C^\infty_c(\R^n,\R)_\cP\}\; .
\]
To this end we proceed as in \cref{case cc dim N}, by using the isomorphism of Lie algebra:  
\[X_{\R^{n-1}\times \cP}\in \diff^\infty_c(\R)_{\R^{n-1}\times \cP}\mapsto  Y_\cP \in \sh(\R^n)_\cP\text{ such that }\widehat{X_{\R^{n-1}\times \cP}}=\widehat {Y_\cP}\; ,\]
and using  \cref{case cc dim 1N}. 
\end{proof}

 \begin{appendix}
 \section{Smoothness of outputs} 
 \label{lieu preuve p.diffeocontinuity}
 
   The proof of \cref{p.diffeocontinuity} stating that the output of a plugin is necessarily smooth is similar to the classical renormalization performed by Douady-Ghys \cite{douady1987disques,ghys1984transformations}, Yoccoz \cite{Yo95} and Shilnikov-Turaev \cite{ST00}. 
\begin{proof}[Proof of \cref{p.diffeocontinuity}] 
 Let $g$ be a plugin with step $  \sigma$. Let $\pi: \tilde\T\times M :=\R \times M\to\T\times M$ be the canonical cover. Let $\tilde g  $ be a lift of $g$ such that $\tilde g(0,y)= (\sigma, y)$ for every $y\in M$.
 \begin{fact} The  action $\phi: (k,z)\in \Z\times \tilde\T\times M \mapsto \tilde g^k(z) \in \tilde\T\times M $ is free,  proper and discontinuous.
 \end{fact} 
 \begin{proof} The action is free since no point of $\Delta+\Z$ is fixed by $\tilde g$ nor in its complement (every point must come back to $\Delta+\Z$). It is discontinuous since any $x\in \R\times M$ has its  orbit which equals the one of a certain $z\in [k,k+\sigma)\times M$ for some $k\in \Z$  by \cref{def.plugin}.(iii), and the orbit of $z$ is discrete by \cref{def.plugin}.(i). Finally the action is proper since $\tau $ is bounded by  some $N\ge 1$ by \cref{def.plugin}.(ii), and so any $\tilde g^{N+1}(\theta,y)$ has its $\R$-coordinate greater than $\theta+\sigma$. \end{proof}
 
 Thus the quotient $C:=\tilde\T\times M/\phi$ is a manifold. As $\tilde g$ sends the left hand side of $\Delta$ to its right hand side, the image of $\Delta$ by the group action is both open and closed, hence equal to the connected set $C$. Therefore, $\Delta_\sigma= [0,\sigma)\times M$  is a fundamental domain of this group action. Also the rescaling map $H_\sigma : \Delta_\sigma\to\T\times M$ induces a diffeomorphism between $C$ and $\T\times M$.

\begin{figure}[H]
\begin{center} \def\svgwidth{220pt}           %\'echelle du dessin
 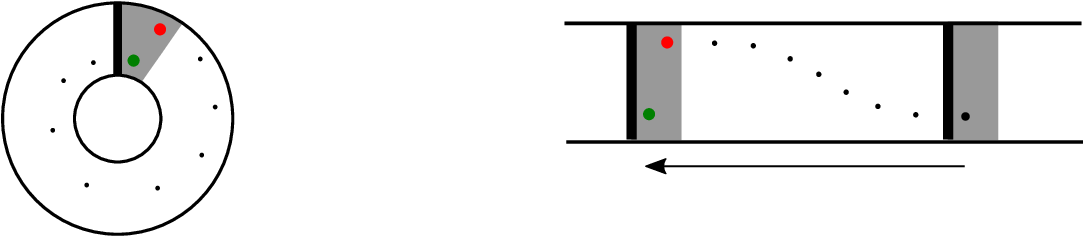
\end{center} 
\end{figure}

Now observe that $\tilde T:(\theta, y)\mapsto (\theta+1,y)$ and 
  $\tilde g$ commutes: $\tilde T\circ \tilde g= \tilde g\circ \tilde T$.
 Thus $\tilde T$ defines a smooth diffeomorphism $T$ on $C$. To determine $T$, we (abusively) identify $\Delta_\sigma$ to a subset of both $\tilde{\T}\times M$ and $\T\times M$. 
  Given  $x\in \Delta_\sigma$, the point $\tilde T(x)\in [1,1+\sigma)\times M$ is equivalent to the point $y\in \Delta_\sigma$ such that there exists $k\ge 0$ satisfying $\tilde g^k(y)=\tilde T(x)$. Note that $k=\tau(y)$ and so $\tilde T(x)= \tilde g^{\tau } (y)$.
Composing by $\pi$, we obtain that $x=   g^{\tau } (y)$.  Hence $T$ is equal to the inverse of the action of $ g^{\tau }$ on $C$. Therefore  $G$ is a diffeomorphism.
 Observe that this construction depends continuously on $g$ and so the output depends continuously on the plugin. As the space of plugins is connected, it comes that the space of outputs is connected 
to $\Id$ by \cref{exem id0}.

\end{proof}

\section{Lie algebras associated to closed subgroups}
\label{a.lie}
 This section is dedicated to show  \cref{p.sppclose} which states that $ \spp $ is a closed Lie algebra. We prove this using general arguments on closed subgroups of $ \Diff_c^\infty ( V) $ where $V$ is a manifold. 
 \begin{proposition}\label{prop lie 1}
For every closed   subgroup $\G \subset \Diff_c^\infty ( V )$, the following is 
 a closed Lie subalgebra of $ \diff_c^\infty ( V) $: 
\begin{equation*}  \sg:= \lbrace X \in \diff_c^\infty ( V) : \Fl_X ^t  \in \G,\;  \forall t \in\R \rbrace .\end{equation*} 
\end{proposition} 
We immediately deduce the result of \cref{p.sppclose} by applying the latter proposition  with $ V = \T \times  M $ and $  \G = \P$. 
 \begin{proof}[Proof of \cref{prop lie 1}]
\underline{$\sg$ is a vector space.} First note that if $ X \in\G $ and $ \lambda \in \R$, it holds ${ \lambda \cdot X \in \sg}$. Now for $X, Y \in \sg$, for any large integer  $ N $  and $r\ge 1$, observe that: \[ \Fl_{X + Y}^{1/ N} =  \Fl_{X}^{1/N}\circ\Fl_Y^{1/N} +O(N^{-2})\] for the $C^r$-norm. Thus we have:
\begin{equation*} Fl_{X+Y}^1= \left(Fl_{X}^{1/N}\circ\Fl_Y^{1/N}\right)^N+O_{C^r}(N^{-1}),\end{equation*} 
and the supports of $ \left(Fl_{X}^{1/N}\circ\Fl_Y^{1/N}\right)^N$ are included in the union of those of $X$ and $Y$. Thus 
$Fl_{X+Y}^1$ is the limit of $ \left(Fl_{X}^{1/N}\circ\Fl_Y^{1/N}\right)^N$ when $ N \to \infty$ in the topology of $\Diff^\infty_c(V)$. As $\G$ is a group, the map $(Fl_{X}^{1/N}\circ\Fl_Y^{1/N})^N $ belongs to $\G$, and since $\G$ is closed the map  $Fl_{X+Y}^1$ also belongs to $\G$.  Also for every $t\in \R$, by replacing $X,Y $ by $(tX,tY)$, we obtain that $Fl_{X+Y}^t=Fl_{tX+tY}^1$ belongs to $\G$. 

\underline{$\sg$ is a Lie algebra.} For $X, Y \in \sg$ and $r\ge 1$, we have for the $C^r$-norm:
\begin{equation*} \Fl_{[X,Y]}^{\tau^2}= [\Fl_{X}^\tau , \Fl_Y^\tau ] +O(\tau ^3)\end{equation*} 
Thus by taking $\tau ^2=1/N$ small we have:
\begin{equation*} \Fl_{[X,Y]}^{1}= \left([\Fl_{X}^{\sqrt N}  , \Fl_Y^{\sqrt N} ]\right)^N +O({\sqrt N}^{-1})\; .\end{equation*} 
So $\Fl_{[X,Y]}^{1}$ is in $\G$. Similarly, we have for any $t$ that $\Fl_{[tX,Y]}^{1}\in \G$, and so $\Fl_{[X,Y]}^{t}=\Fl_{[tX,Y]}^{1}\in \G$.  

\underline{$\sg$ is closed.} As for every $t \ge 0$, the  map $ \Fl^t : X \in  \Diff_c^\infty ( V ) \mapsto\Fl_X^t $ is continuous and $\G$ is closed, the set $  \{ X \in \diff_c^\infty ( V) :\Fl_X ^t  \in \G \} $ is closed. Thus the intersection $\sg$ of the latter sets for all $t$ is closed.  
\end{proof}
To state the parameteric counterpart of the latter proposition,  given a manifold $\cP$, we define:
\begin{equation*} \Fl^t: X_\cP=(X_p)_{p\in \cP} \in  \diff^\infty_c(V)_\cP \mapsto (\Fl^t_{X_p})_{p\in \cP}\in \Diff^\infty_c(V)_\cP\; .\end{equation*} 
Note that the following diagram commutes:
\begin{equation*}  
\begin{array}{rcccl} 
	&					&\mathfrak{inc}		&					&\\
	&\diff^\infty_c (V)_\cP	&\hookrightarrow	&\diff^\infty_c (V\times \cP)&\\
\Fl^t	&\downarrow			&				&\downarrow			&\Fl^t\\
	&\Diff^\infty_c (V)_\cP	&\hookrightarrow	&\Diff^\infty_c (V\times \cP)&\\
	&					&{inc}		&					&\\
\end{array}\end{equation*} 
where $\mathfrak{inc} (X_\cP):= \widehat {X_\cP}$ and $inc(f_\cP):= \widehat {f_\cP}$. 

 \begin{coro}\label{coro lie 1}
For every closed   subgroup $\G_\cP \subset \Diff_c^\infty ( V )_\cP$, the following is 
 a closed Lie subalgebra of $ \diff_c^\infty ( V)_\cP $: 
\begin{equation*}  \sg_\cP:= \lbrace X_\cP \in \diff_c^\infty ( V)_\cP : \Fl_X ^t  \in \G,\;  \forall t \in\R \rbrace .\end{equation*} 
\end{coro} 
\begin{proof} 
First note that  $inc(\G_\cP)$ is a closed Lie sub-group of $\Diff^\infty_c (V\times \cP)$. Hence it define via \cref{prop lie 1} a closed Lie algebra  $\widehat{\sg_\cP}$.  By commutativity of the diagram, we have:
\[ \mathfrak{inc} (\sg_\cP):=  ( \widehat{\sg_\cP}) \; . \]
Hence $\sg_\cP$ is a closed Lie subalgebra of $\diff_c^\infty ( V)_\cP$. 
\end{proof} 
\end{appendix}

\bibliographystyle{alpha}

%\nocite{*}
\bibliography{main.bib}
\end{document}